\documentclass[11pt,a4paper,reqno]{amsart}
\usepackage{graphicx} 
\usepackage{verbatim,amssymb,hyperref,mathrsfs,amsfonts,bbm,stmaryrd}
\usepackage{dsfont}
\usepackage{ytableau}
\usepackage[all]{xy}
\usepackage{wasysym}
\usepackage[utf8]{inputenc}
\usepackage{BOONDOX-calo}
\usepackage{tikz-cd,amscd,pb-diagram}
\usetikzlibrary{arrows,decorations.markings}
\usepackage{imakeidx}
\tikzstyle{dot}=[draw, fill =black, circle, inner sep=0pt, minimum size=2pt]

\hypersetup{
  hidelinks,
  pdftitle={Scopes equivalence for blocks of Ariki-Koike algebras},
  pdfauthor={Joseph Chuang and Kai Meng Tan},
  pdfsubject={Representation theory},
  pdfkeywords={Ariki-Koike algebras, Scopes equivalence, moving vectors, RoCK blocks}
}

\numberwithin{equation}{section}

\theoremstyle{plain}
  \newtheorem{thm}{Theorem}[section]
  \newtheorem{cor}[thm]{Corollary}
  \newtheorem{lem}[thm]{Lemma}
  \newtheorem{prop}[thm]{Proposition}
  
\theoremstyle{definition}
  \newtheorem{Def}[thm]{Definition}
\theoremstyle{remark}
  \newtheorem{rem}[thm]{Remark}
  \newtheorem{eg}[thm]{Example}
  
  \newtheorem*{rem*}{Remark}
\numberwithin{equation}{section}

\newcommand{\bone}{\boldsymbol{1}}
\newcommand{\bbone}{\mathbbm{1}}
\newcommand{\bl}{\boldsymbol{\lambda}}

\newcommand{\bm}{\boldsymbol{\mu}}
\newcommand{\bnu}{\boldsymbol{\nu}}

\newcommand{\UU}{\mathbf{U}}
\newcommand{\uu}{\upsilon}
\newcommand{\fs}{\mathcal{s}}

\newcommand{\ba}{\mathbf{a}}
\newcommand{\br}{\mathbf{r}}
\newcommand{\bs}{\mathbf{s}}
\newcommand{\bt}{\mathbf{t}}
\newcommand{\bu}{\mathbf{u}}
\newcommand{\bv}{\mathbf{v}}

\newcommand{\bx}{\mathbf{x}}
\newcommand{\by}{\mathbf{y}}
\newcommand{\bz}{\mathbf{z}}
\newcommand{\wt}{\mathsf{wt}}
\newcommand{\core}{\mathsf{core}}
\newcommand{\quot}{\mathsf{quot}}
\newcommand{\mv}{\mathbf{mv}}

\renewcommand{\min}{\mathrm{min}}
\renewcommand{\max}{\mathrm{max}}
\newcommand{\ZZ}{\mathbb{Z}}

\newcommand{\cs}{\mathsf{s}}
\newcommand{\HH}{\mathcal{H}}

\newcommand{\BB}{\mathcal{B}}
\newcommand{\B}{\mathsf{B}}
\newcommand{\C}{\mathsf{C}}

\newcommand{\rr}{\rho}
\newcommand{\FF}{\mathbb{F}}

\newcommand{\Sp}{S}
\newcommand{\Hub}{\mathsf{Hub}}
\newcommand{\hub}{\mathsf{hub}}
\newcommand{\PP}{\mathcal{P}}
\newcommand{\res}{\mathsf{res}}

\newcommand{\len}{l}

\newcommand{\ff}{\varphi}

\newcommand{\Be}{\mathbf{e}}
\newcommand{\EW}{\widehat{\mathbf{W}}}
\newcommand{\AW}{\mathbf{W}}
\newcommand{\sym}[1]{\mathfrak{S}_{#1}}
\newcommand{\DDot}[1]{\mathrel{\overset{#1}{\bullet}}}
\newcommand{\CDot}[1]{\mathrel{\overset{#1}{\circ}}}

\newcommand{\AAbar}{\overline{\mathcal{A}}_e^{\ell}}
\newcommand{\bij}{\iota}

\newcommand{\EP}{\pmb{\varnothing}}
\newcommand{\Ht}{\mathsf{ht}}

\newcommand{\II}{\mathcal{I}}

\newcommand{\dd}{\mathsf{d}}

\newcommand{\ft}[2]{\llparenthesis\, #1,#2 \,\rrparenthesis}

\def\Z{{\mathbb Z}}

\makeindex[columns=2]

\title{Scopes equivalence for blocks of Ariki-Koike algebras}

\author{Joseph Chuang}
\address[J. Chuang]{Department of Mathematics, City St George's, University of London, Northampton Square, London EC1V 0HB, United Kingdom.}
\email{joseph.chuang.1@citystgeorges.ac.uk}

\author{Kai Meng Tan}
\address[K. M. Tan]{Department of Mathematics, National University of Singapore, Block S17, 10 Lower Kent Ridge Road. Singapore 119076.}
\email{tankm@nus.edu.sg}

\date{August 2026}

\thanks{This work was supported by Singapore Ministry of Education AcRF Tier 1 A-8003568-00-00.}


\begin{document}

\maketitle

\begin{abstract}
We obtain necessary and sufficient conditions for a block of an Ariki-Koike algebra to have the property that all its associated multipartitions have no addable node with a given residue.
This leads to a classification of Scopes equivalence classes for Ariki-Koike algebras in terms of their pyramid numbers and Scopes vectors, generalising that for level $1$ and for core blocks.
Our proof is independent of the results of Richards \cite{R} for level 1.
\end{abstract}

\section{Introduction}

In her proof of Donovan's conjecture for symmetric groups---that there are only a finite number of Morita inequivalent blocks with a given defect group---Scopes proved \cite{Scopes} a pivotal result: when two blocks form a so-called `$[w:k]$-pair' with $k \geq w$, they are Morita equivalent, and the correspondence between their modules can be described in a straightforward  combinatorial way.
The equivalence relation generated by this relation is now commonly known as {\em Scopes equivalence}.

Analogous Scopes equivalence and $[w:k]$-pairs have since been developed for Iwahori-Hecke algebras of type A and $q$-Schur algebras.
In particular, Richards established in \cite{R} a necessary and sufficient condition for two blocks of Iwahori-Hecke algebras of type A---of which blocks of symmetric group algebras are special cases---to be Scopes equivalent, that they must have the same {\em pyramid numbers}.

To explain Scopes equivalence in a more modern language, there is an $\widehat{\mathfrak{sl}}_e$-categorification of the module categories of the Iwahori-Hecke algebras of type A, in which the generators $e_i, f_i$ of $\widehat{\mathfrak{sl}}_e$ act as restriction and induction functors.
This induces an action of the Weyl group $\AW_e$ of $\widehat{\mathfrak{sl}}_e$ on the blocks of these algebras, with each orbit consisting precisely of the blocks with the same weight.
A $[w:k]$-pair consists of a block $B$ and its image $\cs_i \DDot{} B$ under a simple reflection $\cs_i$ of $\AW_e$.
The blocks in such a pair are Morita equivalent if they are at the end of an $i$-string, with an explicit description of the correspondence between their important modules, such as Specht modules and simple modules.
In effect, Scopes proved that for the case of symmetric groups, the combinatorial condition of $k \geq w$ is both necessary and sufficient for such a $[w:k]$-pair to be at the end of an $i$-string.

For Ariki-Koike algebras, a generalisation of the Iwahori-Hecke algebras of type A, there is also an $\widehat{\mathfrak{sl}}_e$-categorification of their module categories, inducing a $\AW_e$-action on their blocks.
A block $B$ of an Ariki-Koike algebra is also Morita equivalent to $\cs_i \DDot{} B$ if they are at the end of an $i$-string, with a similarly explicit description of the correspondence between important modules.
However, an equally straightforward combinatorial criterion to determine exactly when this happens has been elusive.
While Dell'Arciprete \cite{D}, Amara-Omari and Schaps \cite{AM-S}, Li and the second author \cite{LT}, and Declercq and Jacon \cite{DJ} have each provided sufficient conditions, these conditions are not necessary; see \cite[Example 4.14 and Remark 4.15]{LT}. 

The first key development came in \cite{LQT}, where Li, Qi and the second author used the {\em moving vectors} to study the core blocks of Ariki-Koike algebras in detail.
This groundbreaking concept of moving vectors was first introduced by Li and Qi in \cite{LQ}
for the purpose of studying the representation type of blocks of Ariki-Koike algebras.
They were originally only defined for multipartitions associated with FLOTW multicharges in \cite{LQ}, and then generalised to all multicharges with their properties further developed in \cite{LQT}.
In particular, it is shown that they classify the $\AW_e$-orbits for Ariki-Koike algebras: two blocks of Ariki-Koike algebras lie in the same orbit if and only if they have the same moving vector \cite[Proposition 3.13]{LQT}.
Among other things, a combinatorial necessary and sufficient condition for a core block to be at the end of an $i$-string was established \cite[Lemma 4.8 and Proposition 4.23]{LQT}.  In addition, it is shown that two core blocks in the same $\AW_e$-orbit are Scopes equivalent if and only if they have the same {\em Scopes vector} \cite[Theorem 4.24]{LQT}.

The present paper extends the ideas in \cite{LQT} to all blocks of Ariki-Koike algebras.
In Theorem \ref{T:equiv1}, we give combinatorial necessary and sufficient conditions for a block to be at the end of an $i$-string, generalising the characterisation found in \cite{LQT} for the core blocks.
In Theorem \ref{T:Scopes}, we show that two non-core blocks of Ariki-Koike algebras in the same $\AW_e$-orbit are Scopes equivalent if and only if they have the same {\em Scopes vector} and the same {\em pyramid numbers}.
Roughly speaking, the Scopes vector of a block records the ordered residue data of the multicharge of its rank-level dual, while the pyramid numbers encode pairwise differences in that multicharge.
Combining Theorem \ref{T:Scopes} with \cite[Theorem 4.24]{LQT}, we derive Corollary \ref{C: }, which provides the invariants---namely moving vectors, Scopes vectors and pyramid numbers---for classifying the Scopes equivalence classes for Ariki-Koike algebras.  
This classification naturally generalises both Richards's for Iwahori-Hecke algebras and that for core blocks of Ariki-Koike algebras. 
As an application, we show that the RoCK blocks of Ariki-Koike algebras are precisely the blocks whose pyramid numbers all achieve the maximum value, and determine the number of Scopes inequivalent RoCK blocks in a given $\AW_e$-orbit.

We now indicate the layout of this paper.
In the next section, we provide the necessary combinatorial background of (multi-)partitions and (multi-)$\beta$-sets.
In Section \ref{S:main}, we first review the relevant aspects of the representation theory of Ariki-Koike algebras in the first subsection.  
We then establish the two main theorems of this paper, namely Theorems \ref{T:equiv1} and \ref{T:Scopes}, in the second subsection. 
The proof of Theorem \ref{T:Scopes} depends on Proposition \ref{P:unique}, whose technical combinatorial proof is deferred to the next subsection.
In the final subsection, we apply our results to RoCK blocks to obtain a simple description of these blocks in terms of their pyramid numbers, and a closed formula for the number of Scopes inequivalent RoCK blocks in a given $\AW_e$-orbit.

Throughout this paper, we use the following notations and conventions.
Firstly, $\mathbb{Z}^+$ denotes the set of positive integers.  
For $k \in \ZZ$, we write $\ZZ_{< k} := \{ x \in \ZZ \mid x < k \}$, and define similarly $\ZZ_{\leq k}$, $\ZZ_{> k}$ and $\ZZ_{\geq k}$.

We fix $e,\ell \in \mathbb{Z}^+$ with $e \geq 2$. For $a,b \in \ZZ$ and $m \in \{e,\ell\}$,
$$[a,\, b] := \{ x \in \ZZ \mid a \leq x \leq b \},$$
and we write $a \equiv_m b$ for $a \equiv b \pmod m$.
We also identify $\ZZ/m\ZZ$ with $[0,\, m-1]$, and for $\bt = (t_1,\dotsc, t_m) \in \ZZ^m$, write
$$|\bt| := \sum_{i=1}^m t_i.$$

Finally, $\bbone_{\mathtt{p}}$ denotes the indicator function on the statement $\mathtt{p}$ and $\delta_{ij}$ denotes the usual Kronecker delta for $i,j \in \ZZ$, i.e.\ $$\bbone_{\mathtt{p}} = \begin{cases}
1, &\text{if $\mathtt{p}$ is true};\\
0, &\text{otherwise},
\end{cases}
\qquad \qquad
\delta_{ij} = 
\begin{cases}
1, &\text{if } i = j;\\
0, &\text{otherwise}.
\end{cases}
$$

\section{Combinatorial background}

\subsection{$\beta$-sets and abaci} \label{SS:beta sets}

A \index{$\beta$-set} $\beta$-set $\B$ is a subset of $\Z$ such that both $\max(\B)$ and $\min(\Z \setminus \B)$ exist. Its \index{charge} \index{charge!$\fs(\B)$} charge $\fs(\B)$ is defined as
$$\fs(\B) := |\B \cap \Z_{\geq 0}| - |\Z_{<0} \setminus \B|.$$
We denote the set of all $\beta$-sets by $\BB$ and the set of all $\beta$-sets with charge $s$ by $\BB(s)$.  Clearly, $\BB= \bigcup_{s \in \Z} \BB(s)$ (disjoint union).

Let $\B \in \BB$. The \index{abacus!$\infty$-abacus} $\infty$-abacus display of $\B$ is a horizontal line with positions labelled by $\Z$ in an ascending order going from left to right, with a bead at the position $a$ for each $a \in \B$.
The $e$-abacus display of $\B$ is obtained by cutting its $\infty$-abacus display into sections $\{ ae, ae+1, \dotsc, ae + e-1 \}$ ($a \in \Z$), and putting the section $\{ ae, ae+1, \dotsc, ae + e-1 \}$ directly on top of $\{ (a+1)e, (a+1)e+1, \dotsc, (a+1)e + e-1 \}$.
Thus the \index{abacus!$e$-abacus} $e$-abacus has $e$ infinitely long vertical runners, which we label as $0, 1,\dotsc, e-1$ going from left to right, and countably infinitely many horizontal rows, labelled by $\Z$ in an ascending order going from top to bottom, and the position on row $a$ and runner $i$ equals $ae + i$.
The {\em $e$-quotient of $\B$}, denoted $\quot_e(\B) = (\B_1,\dotsc, \B_{e})$, can be obtained as follows: for each $i \in \ZZ/e\ZZ$, treat each runner $i$ as an $\infty$-abacus, with the position $ae+i$ relabelled as $a$, and obtain the subset $\B_{i+1}$ from the beads in runner $i$ of the $e$-abacus of $\B$.  Formally,
$$
\B_{i+1} = \{ a \in \Z \mid ae+i \in \B \} = \{ \tfrac{x-i}{e} \mid x \in \B,\ x \equiv_e i \} \in \BB.
$$
Its {\em $e$-core}, denoted $\core_e(\B)$, can be obtained by repeatedly sliding its beads in its $e$-abacus up their respective runners to fill up the vacant positions above them, and its {\em $e$-weight}, denoted $\wt_e(\B)$, is the total number of times the beads move one position up their runners.

Note that
\begin{align*}
\sum_{i=1}^{e} \fs(\B_i) = \fs(\B) &= \fs(\core_e(\B)).
\end{align*}

Given $\beta$-set $\B$ and $k \in \Z$, define
$$
\B^{+k} :=  \{ x + k \mid x \in \B \}.
$$
Then $\B^{+k}$ is a $\beta$-set with $\fs(\B^{+k}) = \fs(\B) + k$.  
Furthermore, if $\quot_e(\B) = (\B_1,\B_2,\dotsc, \B_{e})$, then $\quot_e(\B^{+1}) = ((\B_{e})^{+1}, \B_1,\dotsc, \B_{e-1})$.
Consequently,
\begin{equation} \label{E:quot}
\quot_e(\B^{+(ae+b)}) = ((\B_{e-b+1})^{+(a+1)},\dotsc, (\B_{e})^{+(a+1)},(\B_1)^{+a}, \dotsc, (\B_{e-b})^{+a})
\end{equation}
for all $a \in \ZZ$ and $b \in \ZZ/e\ZZ$.

For each $i \in \Z/e\Z$ and $\B \in \BB$, define $\hub_i(\B) = \hub_{i,e}(\B)$ by
\begin{align*}
\hub_i(\B) &:= |\{ x \in \B \mid x\equiv_e i,\, x-1 \notin \B \}| - |\{ x \in \B \mid x \equiv_e i-1,\, x+1 \notin \B \}|
\end{align*}
Furthermore, call the $e$-tuple $(\hub_0(\B), \dotsc, \hub_{e-1}(\B))$ the {\em $e$-hub of $\B$}, denoted $\Hub_e(\B)$.
By \cite[Lemma 2.1]{LQT}, if $\quot_e(\B) = (\B_1,\dotsc, \B_{e})$, then
$
\hub_j(\B) =
\fs(\B_{j+1}) - \fs(\B_{j}) - \delta_{j0} 
$ 
for all $j \in \ZZ/e\ZZ$,
where
$\B_0 := \B_e$.
In particular,
\begin{equation}
\Hub_e(\B) = \Hub_e(\core_e(\B)). \label{E:hub-core}
\end{equation}


\subsection{Partitions}
A partition is a weakly decreasing infinite sequence of nonnegative integers that is eventually zero.  Given a partition $\lambda = (\lambda_1,\lambda_2,\dotsc)$, we write \begin{align*}
\len(\lambda) &:= \max (\{ 0\} \cup \{ a \in \mathbb{Z}^+ : \lambda_a >0 \}) ; \\
|\lambda| &:= \sum_{a=1}^{\len(\lambda)} \lambda_a.
\end{align*}
When $|\lambda| = n$, we say that $\lambda$ is a partition of $n$.
Its Young diagram $[\lambda]$ is defined as
$$ [\lambda] = \{ (i,j) \in \mathbb{Z} \mid 1 \leq i \leq \len(\lambda), 1 \leq j \leq \lambda_i \}.
$$
We write $\PP(n)$ for the set of partitions of $n$ and $\PP$ for the set of all partitions.  The unique partition of $0$, also known as the empty partition, shall be denoted as $\varnothing$.

Let $\lambda = (\lambda_1,\lambda_2,\dotsc)\in \PP$. For each $t \in \mathbb{Z}$,
$$
\beta_t(\lambda) := \{ \lambda_i+t-i \mid i \in \mathbb{Z}^+\}
$$
is a $\beta$-set, with $\fs(\beta_t(\lambda)) = t$.  We call $\beta_t(\lambda)$ the $\beta$-set of $\lambda$ with charge $t$.
The map $\beta : \PP \times \ZZ \to \BB$, defined by $\beta(\lambda,t) = \beta_t(\lambda)$, is a bijection, and for each $t \in \mathbb{Z}$, $\beta_t = \beta(-,t)$ gives a bijection between $\PP$ and $\BB(t)$.  In particular, when $\B$ is a $\beta$-set, we shall write $\beta^{-1}(\B)$ for the unique partition $\lambda$ satisfying $\beta (\lambda, \fs(\B)) = \beta_{\fs(\B)}(\lambda) = \B$.

The $e$-weight $\wt_e(\lambda)$ and the $e$-core $\core_e(\lambda)$ of  $\lambda$ can be unambiguously defined to be $\wt_e(\beta_t(\lambda))$ and $\beta^{-1}(\core_e(\beta_t(\lambda)))$ respectively for any $t \in \mathbb{Z}$.

We note that if $\B$ is a $\beta$-set with $\quot_e(\B) = (\B_{1},\dotsc, \B_{e})$, then $\core_e(\B)$ is the unique $\beta$-set whose $e$-quotient is $(\beta_{\fs(\B_1)}(\varnothing), \dotsc, \beta_{\fs(\B_{e})} (\varnothing))$ and $\wt_e(\B) = \sum_{i=1}^{e} |\beta^{-1}(\B_{i})|$.

\subsection{Multipartitions}

A multipartition is a finite tuple of partitions, and when such a tuple has length $\ell$, we call it an $\ell$-partition.
An {\em $\ell$-partition of $n$} is an element $\bl = (\lambda^{(1)}, \dotsc, \lambda^{(\ell)}) \in \PP^{\ell}$ such that
$$
|\bl| := \sum_{i=1}^{\ell} |\lambda^{(i)}| = n.$$
We write $\PP^{\ell}(n)$ for the set of all $\ell$-partitions of $n$.
The unique $\ell$-partition of $0$ shall be denoted $\EP$.  Of course, $\EP$ depends on $\ell$, which will always be clear in the context in this paper.

Let $\bl \in \PP^{\ell}$.  The \index{Young diagram} {\em Young diagram of $\bl$} is
$$[\bl] = \{ (a,b,i) \in (\Z^+)^3 \mid a \leq \len(\lambda^{(i)}),\ b \leq \lambda^{(i)}_a,\ i \leq \ell \}.$$
Elements of $[\bl]$ are called {\em nodes}.

We usually associate $\bl$ with an {\em $\ell$-charge} $\bt \in \Z^{\ell}$, and we write \index{$(\bl; \bt)$} $(\bl; \bt)$ for such a pair.  The set of such pairs is then naturally identified with $\PP^{\ell} \times \Z^{\ell}$.

Let $\bt = (t_1,\dotsc, t_{\ell}) \in \Z^{\ell}$.
Given $(a,b,i) \in [\bl]$, its {\em $(e,\bt)$-residue}, denoted $\res_e^{\bt}(a,b,i)$,
is the residue class of $b-a + t_i$ 
modulo $e$. 
We write $\res^{\bt}_e(\bl)$ for the multiset $\{\res_e^{\bt}(a,b,i) \mid (a,b,i) \in [\bl]\}$.

If removing a node $\mathfrak{n}$ of $[\bl]$ from $[\bl]$ yields a Young diagram $[\bm]$ for some $\ell$-partition $\bm$, we call $\mathfrak{n}$ {\em a removable node of $\bl$} and {\em an addable node of $\bm$}.
If $(a,b,i)$ is a removable node of $\bl$, then $b = \lambda^{(i)}_a > \lambda^{(i)}_{a+1}$, and so
such a removable node corresponds bijectively to an $x \in \beta_{t_i}(\lambda^{(i)})$  with $x-1 \notin \beta_{t_i}(\lambda^{(i)})$ (namely, $x = b-a+t_i$).
Note that if $\res_e^{\bt} (a,b,i) = j$, then $x 
\equiv_e \res^{\bt}_e(a,b,i) = j$, so that $x$ lies in runner $j$ of the $e$-abacus display of $\beta_{t_i}(\lambda^{(i)})$.
In the same way, an addable node $(a',b',i')$ of $[\bl]$ corresponds bijectively to an $x$ such that $x-1 \in \beta_{t_{i'}}(\lambda^{(i')})$ but $x \notin \beta_{t_{i'}}(\lambda^{(i')})$ (namely, $x = b'-a' +t_{i'}$).

For $\bl = (\lambda^{(1)},\dotsc, \lambda^{(\ell)}) \in \PP^{\ell}$ and $\bt = (t_1,\dotsc, t_{\ell}) \in \Z^{\ell}$, we define
$$
\hub_j(\bl;\bt) := \sum_{i=1}^{\ell} \hub_j(\beta_{t_i}(\lambda^{(i)}))
$$
for each $j \in \Z/e\Z$.
This is the number of removable nodes of $\bl$ with $(e,\bt)$-residue $j$ minus the number of addable nodes of $\bl$ with $(e,\bt)$-residue $j$.
The {\em $e$-hub of $(\bl;\bt)$}, denoted $\Hub_e(\bl;\bt)$, is then
$$
\Hub_e(\bl;\bt) := (\hub_0(\bl;\bt), \dotsc, \hub_{e-1}(\bl;\bt)).
$$

For $\bl = (\lambda^{(1)},\dotsc, \lambda^{(\ell)}) \in \PP^{\ell}$ and $\bt = (t_1,\dotsc, t_{\ell}) \in \Z^{\ell}$, we shall use the shorthand
$$\beta_{\bt}(\bl) := (\beta_{t_1}(\lambda^{(1)}), \dotsc, \beta_{t_{\ell}}(\lambda^{(\ell)})) \in \BB^{\ell}.$$

\subsection{Uglov's map} \label{SS:Uglov}

For each $i \in [1,\,\ell]$, define $\uu_i\ (= \uu_{e,\ell, i}) : \mathbb{Z} \to \mathbb{Z}$ by
$$
\uu_i(x) = (x - \bar{x})\ell + (\ell-i)e + \bar{x}.
$$
Here $\bar{x}$ is defined by $x \equiv_e \bar{x} \in \{ 0,1,\dotsc, e-1\}$.   In other words, if $x = ae + b$, where $a,b \in \mathbb{Z}$ with $0 \leq b < e$, then $$\uu_i(x) = ae\ell + (\ell-i)e + b.$$

In \cite{Uglov}, Uglov defines a map $\UU\ (= \UU_{e,\ell}) : \BB^{\ell} \to \BB$ as follows:
$$
\UU(\B^{(1)},\dotsc,\B^{(\ell)}) := \bigcup_{j=1}^{\ell}
\uu_j(\B^{(j)}) \qquad (\B^{(1)},\dotsc, \B^{(\ell)} \in \BB).
$$
The best way to understand $\UU$ is via the abaci.
First we stack $\infty$-abacus displays of $\B^{(j)}$'s on top of each other, with the display of $\B^{(1)}$ at the bottom and that of $\B^{(\ell)}$ at the top.
Next, we cut up this stacked $\infty$-abaci into sections with positions $\{ ae, ae+1, \dotsc, ae + e-1 \}$ ($a \in \Z$), and put the section with positions $\{ ae, ae+1, \dotsc, ae + e-1 \}$ on top of that with positions $\{ (a+1)e, (a+1)e+1, \dotsc, (a+1)e + e-1 \}$ (thus the section of the $\infty$-abacus of $\B^{(1)}$ with positions $\{ ae, ae+1, \dotsc, ae + e-1 \}$ is placed directly on top of the section of the $\infty$-abacus of $\B^{(\ell)}$ with positions $\{ (a+1)e, (a+1)e+1, \dotsc, (a+1)e + e-1 \}$), and relabel the rows of this $e$-abacus by $\Z$ in ascending order, with $0$ labelling the section of the $\infty$-abacus of $\B^{(\ell)}$ with positions $\{ 0, 1, \dotsc, e-1 \}$.
This yields the $e$-abacus display of $\UU(\B^{(1)},\dotsc,\B^{(\ell)})$.

For each $(\bl;\bt) \in \PP^{\ell} \times \ZZ^{\ell}$, define
$$
\core_e(\bl;\bt) := \core_e(\UU(\beta_{\bt}(\bl))) \qquad \text{and} \qquad
\wt_e(\bl;\bt) = \wt_e(\UU(\beta_{\bt}(\bl))).
$$

Post-composing $\UU$ by $\quot_e$ gives a bijection from $\BB^{\ell}$ to $\BB^{e}$.
By intertwining with the bijection $\beta$ between $\PP \times \ZZ$ and $\BB$, this further induces a bijection $\bij_{\ell,e} : \PP^{\ell} \times \ZZ^{\ell} \to \PP^e \times \ZZ^e$ defined by $\bij_{\ell,e}(\bl;\bt) = (\bm;\bu)$ where $\quot_e(\UU(\beta_{\bt}(\bl))) = \beta_{\bu}(\bm)$.  
Note that 
$\bu$ satisfies
\begin{equation} \label{E:bu}
\quot_e(\core_e(\bl;\bt)) = \beta_{\bu}(\EP),
\end{equation}
so that it is completely determined by $\core_e(\bl;\bt)$.
The {\em moving vector} of $(\bl;\bt)$, denoted $\mv_e(\bl;\bt)$, is the $\ell$-tuple $(m_1,\dotsc, m_{\ell}) \in (\ZZ_{\geq 0})^{\ell}$, where, for each $i \in [1,\,\ell]$, $m_i$ is the number of nodes of $\bm$ with $(\ell;\bu)$-residue $\ell-i$.
Since $\wt_e(\bl;\bt) = \wt_e(\UU(\beta_{\bt}(\bl))) = |\bm|$, we see that \begin{equation}
|\mv_e(\bl;\bt)| = \wt_e(\bl;\bt). \label{E:wt-mv}
\end{equation}

\begin{lem} \label{L:-}
Let $\bt = (t_1,\dotsc, t_{\ell}) \in \ZZ^{\ell}$ with $t_1 \leq \dotsb \leq t_{\ell}$. 
Let $\bl \in \PP^{\ell}$ with $\mv_e(\bl;\bt) = (m_1,\dotsc, m_{\ell})$
and $\quot_e(\core_e(\bl;\bt)) = \beta_{\bu}(\EP)$.  
For any $\ba \in (\ZZ_{\geq 0})^e$ with $|\ba| \leq \min\{ m_i \mid i \in [1,\, \ell] \}$,
there exists $\bm \in \PP^{\ell}$ such that 
\begin{align*}
\quot_e(\core_e(\bm;\bt- |\ba|\bone)) &= \beta_{\bu - \ell\ba} (\EP), \\
\mv_e(\bm;\bt-|\ba|\bone) &= \mv_e(\bl;\bt) - |\ba|\bone, 
\end{align*}
where $\bone = (1,\dotsc, 1) \in \ZZ^{\ell}$.
\end{lem}

\begin{proof}
We prove by induction on $|\ba|$, where the base case $|\ba| =0$ is trivial.  
For $|\ba| >0$, let $\ba = (a_1,\dotsc, a_e)$ and suppose that $a_j \ne 0$.
Let $(\boldsymbol{\kappa};\bs) \in \PP^{\ell} \times \ZZ^{\ell}$ be such that $\UU(\beta_{\bs}(\boldsymbol{\kappa})) = \core_e(\bl;\bt)$.
Then 
\begin{align*}
s_i &= t_i - m_i + m_{i-1}
\end{align*}
for all $i \in [1,\,\ell]$, where $\bs = (s_1,\dotsc, s_{\ell})$ and $m_0 = m_{\ell}$, by \cite[Corollary 3.5]{LQT}.
Also,
$$ \quot_e(\UU(\beta_{\bs}(\boldsymbol{\kappa}))) = \quot_e(\core_e(\bl;\bt)) = \beta_{\bu}(\EP).$$ 
Let $(\boldsymbol{\kappa}^-; \bs^-) \in \PP^{\ell} \times \ZZ^{\ell}$ be such that
$$
\quot_e(\UU(\beta_{\bs^-}(\boldsymbol{\kappa}^-))) = \beta_{\bu - \ell\,\Be_j}(\EP),
$$
where $\{ \Be_1,\dotsc, \Be_e \}$ denotes the standard $\ZZ$-basis for $\ZZ^e$.
Then $\bs^- = \bs - \bone$.  
Let $\bs^- = (s^-_1,\dotsc, s^-_{\ell})$, and 
\begin{align*}
\bt^{-} = (t^-_1,\dotsc, t^-_{\ell}) &:= \bt - \bone; \\
\mathbf{m}^- = (m^-_1,\dotsc, m^-_{\ell}) &:= \mv_e(\bl;\bt) - \bone.
\end{align*}
Then $t^-_1 \leq \dotsb \leq t^-_{\ell}$ and $\mathbf{m}^- \in (\ZZ_{\geq 0})^{\ell}$.
Since
$$
s^-_i = s_i - 1 = t_i - m_i + m_{i-1} - 1 = t^-_i - m^-_i + m^-_{i-1}$$
for all $i \in [1,\, \ell]$, where $m^-_0 = m^-_{\ell}$,
there exists $\bl^- \in \PP^{\ell}$ such that 
\begin{align*}
\core_e(\bl^-;\bt^-) &= \UU(\beta_{\bs^-}(\boldsymbol{\kappa}^-)), \\
\mv_e(\bl^-;\bt^-) &= \mathbf{m}^-
\end{align*} by \cite[Corollary 3.8]{LQT}.
Then $\quot_e(\core_e(\bl^-,\bt^-)) = \quot_e(\UU(\beta_{\bs^-}(\boldsymbol{\kappa}^-))) = \beta_{\bu-\ell\, \Be_j}(\EP)$.
Let $\ba^- = \ba - \Be_j$; then $\ba^- \in (\ZZ_{\geq 0})^{e}$ and $$|\ba^-| = |\ba| - 1 \leq \min\{ m_i \mid i \in [1,\, \ell] \} - 1 = \min\{ m_i -1\mid i \in [1,\, \ell] \} = \min\{ m^-_i \mid i \in [1,\, \ell] \}.$$ 
Thus, by induction applied to $(\bl^-,\bt^-)$, there exists $\bm \in \PP^{\ell}$ such that 
\begin{align*}
\quot_e(\core_e(\bm;\bt^- - |\ba^-|\bone)) &= \beta_{(\bu - \ell\, \Be_j) - \ell\ba^-}(\EP) = \beta_{\bu - \ell\ba}(\EP), \\
\mv_e(\bm;\bt^- - |\ba^-|\bone) &= \mv_e(\bl^-;\bt^-) - (|\ba^-|)\bone = \mv_e(\bl;\bt) - |\ba|\bone.
\end{align*}
Since $\bt^- - |\ba^-| \bone = \bt - |\ba|\bone$, we are done.
\end{proof}

\subsection{The extended affine Weyl group $\EW_{m}$} \label{SS:Weyl}

Let $m \in \Z^+$.
The symmetric group on $m$ letters is denoted by $\sym{m}$.  We view $\sym{m}$ as a group of functions acting on $\{1,2,\dotsc,m\}$ so that we compose from right to left; for example, $(1,2)(2,3) = (1,2,3)$.  This is a Coxeter group of type $A_{m-1}$, with Coxeter generators $\cs_i = (i,i+1)$ for $1 \leq i \leq m-1$.

Given any nonempty set $X$, $\sym{m}$ has a natural right place permutation action on $X^{m}$ via
$$
(x_1,\dotsc, x_{m})^{\sigma} = (x_{\sigma(1)},\dotsc, x_{\sigma(m)})
$$
for $(x_1,\dotsc, x_m) \in X^m$ and $\sigma \in \sym{m}$.

Let $\{ \Be_1, \dotsc, \Be_m\}$ denote the standard $\Z$-basis for the free (left) $\Z$-module $\Z^m$.
The right place permutation action of $\sym{m}$ on $\Z^m$ induces the semidirect product $\EW_m = \sym{m} \ltimes \Z^m$,
where
\begin{equation*}
(\sigma \bt)(\tau \bu) = (\sigma\tau) (\bt^{\tau} + \bu) = (\sigma\tau) (t_{\tau(1)} + u_1, \dotsc, t_{\tau(m)}+ u_m) 
\end{equation*}
for all $\sigma,\tau \in \sym{m}$ and $\bt,\bu \in \Z^m$ with $\bt = (t_1,\dotsc, t_m)$ and $\bu = (u_1,\dotsc, u_m)$.
This is an extended affine Weyl group of type $A^{(1)}_{m-1}$, and it contains the affine Weyl group $\AW_m$ of type $\widetilde{A}_{m-1}$ generated by $\{ \cs_0,\cs_1, \dotsc, \cs_{m-1}\}$, where $\cs_a = (a,a+1) \in \sym{m}$ for $1 \leq a \leq m-1$ as before, and $\cs_0 = (1,m)(\Be_m - \Be_1)$.
Note that under this realisation,
\begin{equation}
\AW_m = \{ \sigma \bt \in \EW_m \mid \sigma \in \sym{m},\ \bt \in \ZZ^m,\ |\bt| = 0 \}. \label{E:affineWeylgroup}
\end{equation}


Given $w = \sigma \bt \in \EW_m$ where $\sigma \in \sym{m}$ and $\bt \in \ZZ^{m}$, we write $\overline {w}$ for $\sigma$ (thus $\overline{w}$ denotes the projection of $w$ onto $\sym{m}$).
In addition, let $\rr_m \in \sym{m}$ denote the ascending $m$-cycle $(1,2,\dotsc, m)$ in $\sym{m}$.

\begin{lem}[{\cite[Lemma 2.7(1)]{LQT}}] \label{L:AA}
Let $m \in \ZZ^+$ and $\rr_{m} = (1,2,\dotsc, m) \in \sym{m}$.
For $a \in \ZZ$ and $b \in \ZZ/m\ZZ$, we have
$$(\rr_{m}\Be_{m})^{am + b} = \rr_{m}^b \bv_{m}(am + b),$$ where 
$
\bv_{m}(am+b) := (a,\dotsc, a, \underbrace{a+1,\dotsc, a+1}_{b \text{ times}}) \in \ZZ^{m}.
$
\end{lem}

We now focus on the case where $m = \ell$.
We have a right action of $\EW_{\ell}$ on $\Z^{\ell}$ and $\PP^{\ell}$  via
\begin{align*}
\bt^{\sigma \bu} = \bt^{\sigma} + e\bu \quad \text {and} \quad
\bl^{\sigma \bu} = \bl^{\sigma}
\end{align*}
for $\bl \in \PP^{\ell}$, $\bt, \bu \in \Z^{\ell}$ and $\sigma \in \sym{\ell}$, which further induces a right action on $\PP^{\ell}\times \Z^{\ell}$.
We write $x^{\EW_{\ell}}$ for the orbit of $x$ (for $x$ in $\Z^{\ell}$ or $\PP^{\ell}$ or $\PP^{\ell}\times \Z^{\ell}$), and note that $\res^{\bt}_e(\bl)$ and $\Hub_e(\bl;\bt)$ are invariant under this action; i.e.\
\begin{equation*}
\res^{\bt^w}_e(\bl^w) = \res^{\bt}_e(\bl) \quad \text{and} \quad \Hub_e((\bl;\bt)^w) = \Hub_e(\bl;\bt) \label{E:res-hub}
\end{equation*}
for all $w \in \EW_{\ell}$.

\begin{prop} \label{P:abacus}
Let $(\bl;\bt) \in \PP^{\ell} \times \ZZ^{\ell}$, 
and let $\bij_{\ell,e}(\bl;\bt) = (\bm;\bu)$. 
Let $\bm = (\mu^{(1)},\dotsc, \mu^{(e)})$ and $\bu = (u_1,\dotsc, u_e)$, and further write $\C_j$ for $\beta_{u_j}(\mu^{(j)})$ for each $j \in [1,\,e]$.
\begin{enumerate}
\item 
If $\mu^{(j)} \ne \varnothing$ (equivalently, $\min(\ZZ \setminus \C_j) < u_j \leq \max(\C_j)$), then $\mu^{(j)}$ has at least
$$
\left\lfloor \frac{\max(\C_j) -i}{\ell} \right\rfloor  - \left\lfloor \frac{\min(\ZZ \setminus \C_j) -i}{\ell} \right\rfloor
$$
nodes with $(\ell,u_j)$-residue $i$
for all $i \in \ZZ/\ell\ZZ$.

\item Let $j \in \ZZ/e\ZZ$.  The following statements are equivalent:
\begin{enumerate}
\item 
$\bl^{w}$ has an addable node with $(e,\bt^w)$-residue $j$ for all $w \in \EW_{\ell}$.
\item $\bl^w$ has an addable node with $(e,\bt^w)$-residue $j$ for some $w \in \EW_{\ell}$.
\item $\lambda^{(i)}$ has an addable node with $(e,t_i)$-residue $j$ for some $i$.
\item there exists $x \in \C_j$ such that $x + \delta_{j0} \ell \notin \C_{j+1}$, where $\C_{0} := \C_e$.
\end{enumerate}
\end{enumerate}
\end{prop}

\begin{proof} \hfill
\begin{enumerate}
\item Since moving the bead at $x \in \C_j$ to a vacant position $y \in \ZZ \setminus \C_j$ (with $x > y$) corresponds to removing from the Young diagram of $\mu^{(j)}$ a rimhook of size $x-y$, beginning at the top with a node with $(\ell,u_j)$-residue $\equiv_{\ell} x$ and ending at the bottom with a node with $(\ell,u_j)$-residue $\equiv_{\ell} y + 1$, each $b \in [\min(\ZZ \setminus \C_j)+1,\, \max(\C_j)]$ contributes at least one node of $\mu^{(j)}$ with $(\ell,u_j)$-residue $\equiv_{\ell} b$.
Thus $\mu^{(j)}$ has at least
$$
N_i := |\{ b \in [\min(\ZZ \setminus \C_j)+1,\, \max(\C_j)] \mid b \equiv_{\ell} i \}|
$$
nodes with $(\ell,u_j)$-residue $i$.
Now, if $\max(\C_j) - i = q_M \ell + r_M$ and $\min(\ZZ \setminus \C_j) -i = q_m \ell + r_m$ where $q_M,q_m \in \ZZ$ and $r_M, r_m \in \ZZ/\ell\ZZ$, then
$$
\{ b \in [\min(\ZZ \setminus \C_j)+1,\, \max(\C_j)] \mid b \equiv_{\ell} i \}
= \{ a \ell + i \mid a \in [q_m+1,q_M] \}.
$$
Thus,
$$
N_i = |[q_m + 1, q_M]| = q_M - q_m = \left\lfloor \frac{\max(\C_j) -i}{\ell} \right\rfloor  - \left\lfloor \frac{\min(\ZZ \setminus \C_j) -i}{\ell} \right\rfloor
$$
as desired.

\item 
It is easy to see that (a) $\Rightarrow$ (b) $\Rightarrow$ (c) $\Rightarrow$ (a), so we only show (c) $\Leftrightarrow$ (d).

We have seen that $\lambda^{(i)}$ having an addable node with $(e,t_i)$-residue $j$ is equivalent to having a $b \in \beta_{t_i}(\lambda^{(i)})$ with $b \equiv_e j-1$ such that $b+1 \notin \beta_{t_i}(\lambda^{(i)})$.
Let $b = qe + r$ where $q \in \ZZ$ and $r \in \ZZ/e\ZZ$.
Then $r \equiv_e j-1$, and 
\begin{align*}
\uu_i(b) &= qe\ell + (\ell -i)e + r = (q\ell + \ell -i)e + r ; \\
\uu_i(b+1) &= (q+\delta_{r,e-1})e\ell + (\ell-i)e + (r+1 - \delta_{r,e-1}e) \\
&= (q\ell  + \delta_{r,e-1}\ell + \ell-i) e + (r+1 - \delta_{r,e-1}e).
\end{align*}
Thus, $q\ell + \ell -i \in \C_{r+1} = \C_j$ and $q\ell  + \delta_{r,e-1}\ell + \ell-i \notin \C_{r+2} = \C_{j+1}$. Setting $x = q\ell + \ell -i$ shows (c) $\Rightarrow$ (d).  

For the converse, let $i \in [1,\, \ell]$ be such that $i\equiv_{\ell} -x$. 
Then $x \in \C_j$ implies that $$(x+\delta_{j0})e+j-1 = \uu_i(b)$$ for some $b \in \beta_{t_i}(\lambda^{(i)})$.  Now $\uu_i(b+1) = (x+\delta_{j0}\ell)e+j$  so that $x+\delta_{j0}\ell \notin \C_{j+1}$ implies that $b+1 \notin \beta_{t_i}(\lambda^{(i)})$, and the proof is complete.
\end{enumerate}
\end{proof}

\begin{rem}
  Proposition \ref{P:abacus}(1) may be considered as a generalisation of \cite[Lemma 4.2]{LQT}, and forms the basis for the necessary and sufficient condition for two blocks of Ariki-Koike algebras to be Scopes equivalent in this paper.
\end{rem}

An important subset of $\ell$-charges is the following:
\begin{align*}
\index{$\AAbar$} \AAbar &:= \{ (a_1,\dotsc, a_{\ell}) \in \ZZ^{\ell} \mid a_1 \leq a_2 \leq \dotsb \leq a_{\ell} \leq a_1 + e \}.
\end{align*}
In particular, this subset appears in the following result on how the $e$-cores and the $e$-weights of the $\ell$-partitions, with their associated charges, in the same $\EW_{\ell}$-orbit are related.

\begin{thm}[{\cite[Theorem 3.6]{LT}}] \label{T:core-n-weight}
Let $(\bl;\bt) \in \PP^{\ell} \times \ZZ^{\ell}$, and let $(\bm;\bu) \in (\bl;\bt)^{\EW_{\ell}}$.  Then
\begin{enumerate}
  \item $\beta^{-1}(\core_e(\bm;\bu)) = \beta^{-1}(\core_e(\bl;\bt))$;
 \item $\wt_e(\bm;\bu) = \min( \wt_e( (\bl;\bt)^{\EW_{\ell}} ) )$ if and only if $\bu \in \AAbar$.
\end{enumerate}
\end{thm}

We now look at the case where $m = e$.
For each $k \in \ZZ$, we have a left action $\DDot{k}$ of $\EW_e$ on $\Z$ defined by
\begin{alignat*}{2}
\cs_i \DDot{k} x &=
\begin{cases}
x-1, &\text{if } x \equiv_e i \\
x+1, &\text{if } x \equiv_e i-1 \\
x, &\text{otherwise}
\end{cases}
&\qquad & (i \in [ 1,\, e-1]); \\
\Be_j \DDot{k} x &=
\begin{cases}
x+k e, &\text{if } x \equiv_e j-1 \\
x, &\text{otherwise}
\end{cases}
&& (j \in [1,\, e]).
\end{alignat*}

We extend this left action $\DDot{k}$ to the set $\BB$ of $\beta$-sets by $w \DDot{k} \B = \{ w \DDot{k} b \mid b \in \B \}$ ($w\in \EW_e$, $\B \in \BB$),
and
to $\BB^{\ell}$ by $w \DDot{k} (\B_1,\dotsc, \B_{\ell}) = (w\DDot{k}\B_1, \dotsc, w \DDot{k} \B_{\ell})$ ($w\in \EW_e$, $(\B_1,\dotsc, \B_{\ell}) \in \BB^{\ell}$).

For $\bl = (\lambda^{(1)},\dotsc, \lambda^{(\ell)}) \in \PP^{\ell}$ and $\bt = (t_1,\dotsc, t_{\ell}) \in \Z^{\ell}$, we have
a left action $\DDot{\bt}$ of $\AW_e$  on $\PP^{\ell}$,
where for each $j \in \ZZ/e\ZZ$, $\cs_j \DDot{\bt} \bl$ is the $\ell$-partition obtained from $\bl$ by removing all its removable nodes of $(e,\bt)$-residue $j$ and adding all its addable nodes of $(e,\bt)$-residue $j$.
This satisfies
\begin{equation*}
\beta_{\bt}(\cs_j \DDot{\bt} \bl) = \cs_j \DDot{1} \beta_{\bt}(\bl)
\end{equation*}
by \cite[(2.17)]{LT}.

We also have a natural left action $\CDot{k}$ of $\EW_e$ on $\BB^e$ and $\ZZ^e$ as follows:
\begin{align*}
\cs_i \CDot{k} (\B_1,\dotsc, \B_{e}) &= (\B_1,\dotsc,\B_{i-1}, \B_{i+1}, \B_{i}, \B_{i+2},\dotsc, \B_{e}), \\
\Be_j \CDot{k} (\B_1,\dotsc, \B_{e}) &= (\B_1,\dotsc,\B_{j-1}, \B_j^{+k}, \B_{j+1},\dotsc, \B_{e}), \\
\cs_i \CDot{k} (a_1,\dotsc, a_{e}) &= (a_1,\dotsc,a_{i-1}, a_{i+1}, a_{i}, a_{i+2},\dotsc, a_{e}), \\
\Be_j \CDot{k} (a_1,\dotsc, a_{e}) &= (a_1,\dotsc,a_{j-1}, a_j +k, a_{j+1},\dotsc, a_{e}),
\end{align*}
for all $i \in [1,\,e-1]$, $j \in [1,\,e]$, $(\B_1,\dotsc, \B_e) \in \BB^e$ and $(a_1,\dotsc, a_e) \in \ZZ^e$.
Another formulation for $\CDot{k}$ on $\ZZ^e$ is:
\begin{equation} \label{E:CDot}
\sigma \bs \CDot{k} \bt = (k\bs + \bt)^{\sigma^{-1}}
\end{equation}
for $\sigma \in \sym{e}$ and $\bs,\bt \in \ZZ^e$.

Note that $\CDot{k}$ and $\DDot{k}$ are related in the following way:
\begin{align}
w \CDot{k} \quot_e(\B) &= \quot_e( w \DDot{k} \B) \label{E:CDot-tuple-beta-sets}
\end{align}
for all $w \in \EW_e$ and $\B \in \BB$. 





\section{Ariki-Koike algebras} \label{S:main}

\subsection{Preliminaries}

From now on, $\FF$ denotes a fixed field of characteristic $p$, where we allow $p=0$, and we assume that $\FF$ contains a primitive $e$-th root of unity if $e \ne p$.
Let $q \in \FF$ be such that $q = 1$ if $e = p$, and $q$ is a primitive $e$-th root of unity otherwise.
Fix $\br = (r_1,\dotsc, r_{\ell}) \in \Z^{\ell}$, and let $n \in \Z^+$.
The {\em Ariki-Koike algebra} $\HH_n = \HH_{\FF, q, \br}(n)$ is the unital $\FF$-algebra generated by $\{ T_0,\dotsc, T_{n-1} \}$ subject to the following relations:
\begin{alignat*}{2}
(T_0 - q^{r_1})(T_0 - q^{r_2}) \dotsm (T_0 - q^{r_{\ell}}) &= 0; \\
(T_a - q)(T_a+1) &= 0 &\qquad &(a \in [1,\, n-1] ); \\
T_0T_1T_0T_1 &= T_1T_0T_1T_0; && \\
T_aT_{a+1}T_a &= T_{a+1}T_a T_{a+1} &\qquad &(a \in [1,\, n-2]); \\
T_aT_b &= T_bT_a &\qquad &(|a-b| \geq 2).
\end{alignat*}
It is clear from the definition that $\HH_n$ only depends on the orbit $\br^{\EW_{\ell}}$ and not on $\br$.
We adopt the convention that $\HH_0=\FF$, regarded as a unital $\FF$-algebra.

It is well known that $\HH_n$ is a cellular algebra, in the sense of Graham and Lehrer \cite{GL}.
The Specht modules $\Sp^{\bl}$ ($\bl \in \PP^{\ell}(n)$), constructed by Dipper, James and Mathas \cite{DJM}, are known to be cell modules for $\HH_n$.
We note that, unlike $\HH_n$, these Specht modules actually depend on the order of the $r_j$'s.

For $\bl \in \PP^{\ell}$, define its
{\em $\HH$-core}, {\em $\HH$-weight} and {\em $\HH$-hub}, denoted $\core_{\HH}(\bl)$, $\wt_{\HH}(\bl)$ and $\Hub_{\HH}(\bl) = (\hub^{\HH}_0(\bl), \dotsc, \hub^{\HH}_{e-1}(\bl)) $,
by
\begin{align*}
\core_{\HH}(\bl) &= \beta^{-1} (\core_e((\bl;\br)^w)), \\
\wt_{\HH}(\bl) &= \wt_e((\bl;\br)^w), \\
\Hub_{\HH}(\bl) & = \Hub_e((\bl;\br)^w),
\end{align*}
where $w$ is any element of $\EW_{\ell}$ such that $\br^w \in \AAbar$.
By Theorem \ref{T:core-n-weight}, $\core_{\HH}(\bl)$ and $\wt_{\HH}(\bl)$ are well-defined, i.e.\ do not depend on the choice of $w$ such that $\br^w \in \AAbar$, while $\Hub_{\HH}(\bl)$ is well-defined by the definition of $\Hub_e$.

Given a block $B$ of $\HH_n$, we say that an $\ell$-partition $\bl$ {\em lies in $B$} if $\Sp^{\bl}$ lies in $B$.

\begin{thm}[{\cite[Corollary 4.5]{LT}}] \label{T:Naka}
Let $\bl,\bm \in \PP^{\ell}$.  Then $\bl$ and $\bm$ lie in the same block of an Ariki-Koike algebra if and only if $\core_{\HH}(\bl) = \core_{\HH}(\bm)$ and $\wt_{\HH}(\bl) = \wt_{\HH}(\bm)$.
\end{thm}

By Theorem \ref{T:Naka},
we may unambiguously define the core and weight of a block $B$, denoted $\core(B)$ and $\wt(B)$, as the common $\HH$-core and the common $\HH$-weight of the $\ell$-partitions lying in $B$. Moreover, these two parameters completely determine the block.

By \eqref{E:hub-core}, $\HH$-hub is another block invariant of the Ariki-Koike algebras, and we write $\Hub(B) = (\hub_0(B),\dotsc,\hub_{e-1}(B))$ for the common $\HH$-hub of the $\ell$-partitions lying in $B$.

A more important block invariant is the moving vector:

\begin{lem}[{\cite[Lemma 3.10(2)]{LQT}}] \label{L:mv-block invariant}
If $\bl, \bm$ lie in the same block $B$ of $\HH_n$ and $\br' \in \br^{\EW_{\ell}} \cap \AAbar$, then $\mv_e((\bl;\br)^{w}) = \mv_e((\bm;\br)^{w'})$ for all $w,w'\in \EW_{\ell}$ such that $\br^w =\br^{w'} = \br'$.
\end{lem}

We shall thus write $\mv^{\br'}(B)$ for the common moving vector $\mv_e((\bl;\br)^w)$ for $\bl \in \PP^{\ell}$ lying in $B$ and $w \in \EW_{\ell}$ with $\br^w = \br'$.

The moving vector of $B$ is actually well-defined up to cyclic permutation:

\begin{lem}[{\cite[Corollary 3.12]{LQT}}] \label{L:mv-cyclic permutation}
Let $B$ be a block of $\HH_n$, and let $w \in \EW_{\ell}$ such that $\br^{w} \in \AAbar$.  Then 
$$
\{ \mv^{\br'}(B) \mid \br' \in \br^{\EW^{\ell}} \cap \AAbar \} = \{ (\mv^{\br^{w}}(B))^{\rr_{\ell}^i} \mid i \in \ZZ/\ell\ZZ \},
$$
where $\rr_{\ell} = (1,2,\dotsc, \ell) \in \sym{\ell}$.
\end{lem}

Recall that we have a left action $\DDot{\br}$ of $\AW_e$ on $\PP^{\ell}$.
This induces a left action $\DDot{}$ of $\AW_e$ on the set of blocks of Ariki-Koike algebras (with common $\ell$-charge $\br$) such that if $B$ is a block of $\HH_n$, then $\cs_j \DDot{} B$ is the block of $\HH_{n-\hub_j(B)}$ with
$$
\core(\cs_j \DDot{} B ) = \beta^{-1}(\cs_j \DDot{\ell} \beta_{|\br|}(\core(B))) \quad \text{ and } \quad \wt(\cs_j \DDot{} B) = \wt(B)
$$
(see, for example, \cite[Proposition 4.8]{LT}).
The moving vectors classify the orbits of this action:

\begin{prop}[{\cite[Proposition 3.13]{LQT}}] \label{P:moving-vector-Weyl-orbit}
Let $B$ and $C$ be two blocks of Ariki-Koike algebras with the same $\ell$-charge $\br$.
Let $\br' \in \br^{\EW_{\ell}} \cap \AAbar$.
Then $\mv^{\br'}(B) = \mv^{\br'}(C)$ if and only if $B = w \DDot{} C$ for some $w \in \AW_e$.
\end{prop}

\subsection{Scopes equivalence}

Let $B$ and $C$ be blocks of Ariki-Koike algebras.
For $j \in \ZZ/e\ZZ$,
we write \index{$B \xrightarrow{j} C$} $B \xrightarrow{j} C$ if all $\ell$-partitions lying in $B$ have no addable node with $(e,\br)$-residue $j$ and $C = \cs_j \DDot{} B$.
We further say that $B$ and $C$ are {\em Scopes equivalent} if there exists a sequence $B_0,\dotsc, B_k$ of blocks such that 
$B_0= B$, $B_k = C$ and for each $a \in [1,\, k]$, $B_{a-1} \xrightarrow{j_a} B_a$ or $B_{a} \xrightarrow{j_a} B_{a-1}$ for some $j_a \in \ZZ/e\ZZ$.
Clearly, Scopes equivalence is an equivalence relation on the blocks of Ariki-Koike algebras where each equivalence class is a subset of an $\AW_e$-orbit.


The importance of Scopes equivalence is due to the following theorem.

\begin{thm}[see {\cite[Theorem 6.4]{CR}} and {\cite[Lemma 3.1]{Webster}}] \label{T:CR}
Let $B$ be a block of $\HH_n$ and $j \in \Z/e\Z$.  
If $B \xrightarrow{j} \cs_j \DDot{} B$, then $B$ and $\cs_j \DDot{} B$ are Morita equivalent. Under this equivalence, the Specht module $S^{\bl}$ lying in $B$ corresponds to $S^{\cs_j \DDot{\br} \bl}$.
\end{thm}

For a block $B$ of $\HH_n$ and $w \in \EW_{\ell}$ with $\br^w \in \AAbar$, let $\br^{w*}_B \in \ZZ^e$ be such that
\begin{equation*}
\beta_{\br^{w*}_B}(\EP) = \quot_e(\beta_{|\br^w|}(\core(B))). \label{E:multicharge-of-dual-block}
\end{equation*}
Then for each $\bl \in\PP^{\ell}$ lying in $B$, we have
$$
\bij_{\ell,e}((\bl;\br)^w) = (\bl^*,\br^{w*}_B)
$$
for some unique $\bl^* \in\PP^e$ by \eqref{E:bu}.

For the remainder of this subsection, we shall be working towards providing a necessary and sufficient condition for two blocks of Ariki-Koike algebras (with common $\ell$-charge $\br$) to be Scopes equivalent.
Two Scopes equivalent blocks must lie in the same $\AW_e$-orbit, so by Proposition \ref{P:moving-vector-Weyl-orbit} they have the same moving vector.
By Lemma \ref{L:mv-cyclic permutation}, we may assume that the common moving vector of the blocks in this $\AW_e$-orbit has the form $(m_1,\dotsc, m_{\ell})$ with $m_{\ell} = \min \{ m_i \mid i \in [1,\,\ell] \}$.
In other words, we fix a $w_0 \in \EW_{\ell}$ with $\br^{w_0} \in \AAbar$ such that for any block $B$ in the $\AW_e$-orbit, we have $\mv^{\br^{w_0}}(B) = (m_1,\dotsc, m_{\ell})$ with $m_{\ell} = \min \{ m_i \mid i \in [1,\,\ell] \}$.
We can then simplify our notations and write $\mv(B)$ for $\mv^{\br^{w_0}}(B)$, and $\br^*_B$ for $\br^{w_0*}_B$.

With this choice, we make the following definitions, most of which have already been introduced in \cite[Section 4]{LQT}:

\begin{Def} \label{D:}  \hfill
\begin{enumerate}
\item For each $\by = (y_1,\dotsc, y_e) \in \ZZ^e$, define $\sigma_{\by} \in \sym{e}$ by: for each $j \in [1,\,e-1]$,
$$ y_{\sigma_{\by}(j)} \leq y_{\sigma_{\by}(j+1)},$$
 with equality only if $\sigma_{\by}(j) < \sigma_{\by}(j+1)$.
 (In other words, $\sigma_{\by}$ is the element of $\sym{e}$ with minimal length such that $\by^{\sigma_{\by}}$ is weakly increasing.)

\item For $I \subseteq \ZZ_{\geq 0}$ with $0 \in I$ and $b \in \ZZ_{\geq 0}$, define
        $$\Ht_I(b) := \max\{ i \in I \mid i \leq b \}.$$     

\item Let $B$ be a block of $\HH_n$, with $\mv(B) = (m_1,\dotsc, m_{\ell})$.
\begin{enumerate}
\item
Define $(x^B_1,\dotsc, x^B_e), \by^B = (y^B_1,\dotsc, y^B_e) \in \ZZ^e$ and $\bz^B = (z^B_1,\dotsc, z^B_e) \in (\ZZ/\ell\ZZ)^e$ by
    $$
    \br^*_B = (x^B_1,\dotsc, x^B_e) = \by^B \ell + \bz^B.
    $$
    For convenience, we further let $x^B_0 := x^B_e$, $y^B_0 := y^B_e$ and $z^B_0 := z^B_e$.\medskip

\item Write $\sigma_B$ for $\sigma_{\by^B}$
    .

\item
Define
\begin{align*}
m_B &:= \min(\mv(B));  \\
\II_B &:= \{ \ell - i \mid m_i = m_B \}.
\end{align*}
(Note that by our choice of $w_0 \in \EW_{\ell}$ we have $m_{\ell} = m_B$ and thus $0 \in \II_B$.)

\end{enumerate}
\end{enumerate}
\end{Def}

Arising from these definitions, we have the following lemma:

\begin{lem} \label{L:cs_j} \hfill
\begin{enumerate}
\item \cite[Lemma 3.15(2), Lemma 4.7]{LQT}
Let $B$ be a block of $\HH_n$.  Let $j \in \ZZ/e\ZZ$ and let $C = \cs_j \DDot{} B$.
Then
\begin{align*}
\br^*_C &= \cs_j \CDot{\ell} \br^*_B; \\
\by^C & = \cs_j \CDot{1} \by^B = (\by^B)^{\overline{\cs_j}} + \delta_{j0}(\Be_1 - \Be_e); \\
\bz^C &= \cs_j \CDot{0} \bz^B = (\bz^B)^{\overline{\cs_j}}.
\end{align*}
Furthermore, if $\bl \in \PP^{\ell}$ lying in $B$ with $\bij_{\ell,e}((\bl;\br) ^{w_0}) = (\bl^*, \br^*_B)$, we have
$$\bij_{\ell,e}((\cs_j \DDot{\br} \bl; \br)^{w_0}) = ((\bl^*)^{\overline{\cs_j}}; \br^*_C).$$

\item \cite[Proposition 4.13(1)]{LQT} Let $\by= (y_1,\dotsc, y_e) \in \ZZ^e$ and $j \in \ZZ/e\ZZ$.
If $y_{j+1} > y_j + \delta_{j0}$ (where $y_0 = y_e$), then $$\sigma_{\cs_j \smash[t]{\CDot{1}} \by} = \overline{\cs_j}\sigma_{\by}.$$

\end{enumerate}
\end{lem}

We record the following easy corollary arising from Lemma 3.7(1).

\begin{cor} \label{C:w-z}
Let $B$ be a block of $\HH_n$, and let $C = w \DDot{} B$ where $w \in \AW_e$.
Then
$\bz^C = (\bz^B)^{\overline{w}^{-1}}$.

In particular, $\bz^C$ is a rearrangement of $\bz^B$.
\end{cor}

With the definitions/notations in Definition \ref{D:}, we can now state our first main theorem---two separate necessary and sufficient conditions for a block $B$ of $\HH_n$ to have the property that all the $\ell$-partitions lying in it have no addable node with $(e,\br)$-residue $j$:

\begin{thm} \label{T:equiv1}
Let $B$ be a block of $\HH_n$, and let 
$j \in \ZZ/e\ZZ$.  The following statements are equivalent:
\begin{enumerate}
\item There exists $x \in [x_j^B + (m_B + \delta_{j0})\ell,\ x^B_{j+1} ]$ such that $x \equiv_{\ell} i$ for some $i \in \II_B$.

\item $y^B_{j+1} - y^B_j - \delta_{j0} + \bbone_{z^B_j \leq \Ht_{\II_B}(z^B_{j+1})} \geq m_B+1$.

\item $B \xrightarrow{j} \cs_j \DDot{} B$, i.e.\
all $\ell$-partitions lying in $B$ have no addable node with $(e,\br)$-residue $j$.

\end{enumerate}
\end{thm}

\begin{proof} \hfill
\begin{description}
\item[$(1) \Leftrightarrow (2)$]
Since $x^B_{j+1} = y^B_{j+1} \ell + z^B_{j+1}$ and $0 \in \II_B$, we have
$$
\max\{ x \leq x^B_{j+1} \mid x \equiv_{\ell} i \text{ for some } i \in \II_B \} = y^B_{j+1} \ell + \Ht_{\II_B}(z^B_{j+1}).
$$
Thus
\begin{align*}
\exists x \in [&x_j^B + (m_B + \delta_{j0})\ell,\ x^B_{j+1} ], \exists i \in \II_B,\ x\equiv_{\ell} i \\
&\Leftrightarrow
\quad y^B_{j+1} \ell + \Ht_{\II_B}(z^B_{j+1}) \in [x_j^B + (m_B + \delta_{j0})\ell,\ x^B_{j+1} ] \\
&\Leftrightarrow
\quad y^B_{j+1} \ell + \Ht_{\II_B}(z^B_{j+1}) \geq x_j^B + (m_B + \delta_{j0})\ell \\
&\Leftrightarrow
\quad y^B_{j+1} \ell + \Ht_{\II_B}(z^B_{j+1}) \geq y_j^B \ell + z^B_j + (m_B + \delta_{j0})\ell \\
&\Leftrightarrow
\quad y^B_{j+1} - y^B_j - \delta_{j0} \geq m_B + (z^B_j -  \Ht_{\II_B}(z^B_{j+1}))/\ell \\
&\Leftrightarrow
\quad y^B_{j+1} - y^B_j - \delta_{j0} \geq
\begin{cases}
m_B + 1 &\text{if } z^B_j >  \Ht_{\II_B}(z^B_{j+1}) \\
m_B &\text{if } z^B_j \leq  \Ht_{\II_B}(z^B_{j+1})
\end{cases} \\
&\Leftrightarrow
\quad y^B_{j+1} - y^B_j - \delta_{j0} + \bbone_{z^B_j \leq  \Ht_{\II_B}(z^B_{j+1})} \geq m_B +1.
\end{align*}

\item[$(2) \Rightarrow (3)$]
Let $i = \Ht_{\II_B}(z^B_{j+1})$.  
Then $i \in \II_B$, so that $m_{\ell-i} = m_B$, and $0 \leq i \leq z^B_{j+1} \leq \ell-1$.  Consequently,
\begin{gather}
  \left\lfloor \frac{x^B_{j+1} - i}{\ell} \right\rfloor = \left\lfloor \frac{y^B_{j+1}\ell + z^B_{j+1} - i}{\ell} \right\rfloor = y^B_{j+1}, \label{E:y^B_{j+1}} \\
  z^B_j - \bbone_{z^B_j > i}\, \ell \leq i \leq z^B_{j+1}.  \label{E:z^B_j}
\end{gather}

Let $\bl$ be an $\ell$-partition lying in $B$, and let $\iota_{\ell,e}((\bl;\br)^{w_0}) = (\bl^*;\br^*_B)$. Let $\bl^* = (\lambda^{*(1)},\dotsc, \lambda^{*(e)})$ and $\beta_{\br^*_B}(\bl^*) = (\C_1,\dotsc, \C_e)$.
For convenience, we also let $\C_0 := \C_e$ and $\lambda^{*(0)} := \lambda^{*(e)}$.
Note that 
\begin{equation} \label{E:max(C_j)}
\max(\C_j) <
\left( \left \lfloor \frac{\max(\C_j) - i}{\ell} \right \rfloor + 1  \right) \ell + i.
\end{equation}

We claim that $\max(\C_j) + \delta_{j0} \ell < \min (\ZZ \setminus \C_{j+1})$.
Assuming the claim, then $b \in \C_j$ implies that $b + \delta_{j0} \ell \in \C_{j+1}$, so that $\bl$ has no addable node with $(e,\br)$-residue $j$ by Proposition \ref{P:abacus}(2).

It remains to prove the claim.  We consider four cases separately.

\begin{description}
  \item[Case 1. $\lambda^{*(j)} = \varnothing = \lambda^{*(j+1)}$]
  In this case $\C_j = \ZZ_{< x^B_j}$ and $\C_{j+1} = \ZZ_{<x^B_{j+1}}$.
  Thus,
  \begin{align*}
    \max(\C_j) + \delta_{j0}\ell &=  x^B_j -1 + \delta_{j0}\ell = (y^B_j + \delta_{j0})\ell + z^B_{j} - 1 
    \\
    &\leq (y^B_{j+1} - \bbone_{z_j^B > i} - m_B)\ell + z^B_j -1 \\
    &\leq y^B_{j+1} \ell + z^B_{j+1} -1 = x^B_{j+1} -1 < x^B_{j+1} = \min(\ZZ \setminus \C_{j+1}) ,
  \end{align*}
  where the first inequality of the last line follows from \eqref{E:z^B_j}.

  \item[Case 2. $\lambda^{*(j)} = \varnothing \ne \lambda^{*(j+1)}$]
  Similar to Case 1, we obtain $\C_j = \ZZ_{< x^B_j}$, and 
  $$
  \max(\C_j) + \delta_{j0}\ell 
  \leq (y^B_{j+1} - \bbone_{z_j^B > i} - m_B)\ell + z^B_j - 1.$$
  By Proposition \ref{P:abacus}(1) and \eqref{E:y^B_{j+1}}, we have
  \begin{align*}
  m_B = m_{\ell - i}
  &\geq \left \lfloor \frac{\max(\C_{j+1}) - i}{\ell} \right\rfloor - \left \lfloor \frac{\min(\ZZ\setminus \C_{j+1}) -i}{\ell} \right\rfloor \\
  &\geq \left \lfloor \frac{x^B_{j+1} - i}{\ell} \right\rfloor - \left \lfloor \frac{\min(\ZZ\setminus \C_{j+1}) -i}{\ell} \right\rfloor \\
  &= y^B_{j+1} - \left \lfloor \frac{\min(\ZZ\setminus \C_{j+1}) -i}{\ell} \right\rfloor.
  \end{align*}
  Combining the above two inequalities, coupled with \eqref{E:z^B_j}, we get
  \begin{align*}
    \max(\C_j) + \delta_{j0}\ell &\leq \left(\left \lfloor \frac{\min(\ZZ\setminus \C_{j+1}) -i}{\ell} \right\rfloor - \bbone_{z^B_j > i} \right) \ell + z_j^B -1
    < \min(\ZZ \setminus \C_{j+1}).
  \end{align*}

  \item[Case 3. $\lambda^{*(j)} \ne \varnothing = \lambda^{*(j+1)}$]
  We have $\C_{j+1} = \ZZ_{<x^B_{j+1}}$, and by Proposition \ref{P:abacus}(1)
  \begin{align*}
  m_B = m_{\ell - i}
  &\geq \left \lfloor \frac{\max(\C_{j}) - i}{\ell} \right\rfloor - \left \lfloor \frac{\min(\ZZ\setminus \C_{j}) -i}{\ell} \right\rfloor \\
  &\geq \left \lfloor \frac{\max(\C_{j}) - i}{\ell} \right\rfloor - \left \lfloor \frac{x^B_j -1 -i}{\ell} \right\rfloor \\
  &= \left \lfloor \frac{\max(\C_{j}) - i}{\ell} \right\rfloor - y^B_j + \bbone_{z^B_j \leq i}.
  \end{align*}
  Thus, by \eqref{E:max(C_j)}, we have
  \begin{align*}
  \max(\C_j) + \delta_{j0} \ell
  &< 
    \left( \left \lfloor \frac{\max(\C_j) - i}{\ell} \right \rfloor + 1 + \delta_{j0} \right) \ell + i \\
  &\leq (y^B_j - \bbone_{z^B_j \leq i} + m_B + 1 + \delta_{j0} ) \ell + i \\
  &\leq y^B_{j+1} \ell + i \qquad \qquad (\text{by (2)---note that $i= \Ht_{\II_B}(z^B_{j+1})$}) \\
  &
  \leq y^B_{j+1}\ell + z^B_{j+1} = x^B_{j+1} = \min(\ZZ \setminus \C_{j+1}).
  \end{align*}

  \item[Case 4. $\lambda^{*(j)} , \lambda^{*(j+1)} \ne \varnothing$]
  By Proposition \ref{P:abacus}(1),
  \begin{align*}
  m_B = m_{\ell - i}
  &\geq \sum_{k=j}^{j+1} \left( \left \lfloor \frac{\max(\C_{k}) - i}{\ell} \right\rfloor - \left \lfloor \frac{\min(\ZZ\setminus \C_{k}) -i}{\ell} \right\rfloor \right).
  \end{align*}
  This yields
  \begin{align*}
  \min(\ZZ \setminus \C_{j+1})
  &\geq
  \left\lfloor  \frac{\min(\ZZ\setminus \C_{j+1}) -i}{\ell} \right\rfloor \ell + i \notag \\
  &\geq
  \left( \left\lfloor \frac{\max(\C_{j}) - i}{\ell} \right\rfloor - \left\lfloor \frac{\min(\ZZ\setminus \C_{j}) -i}{\ell} \right\rfloor + \left\lfloor \frac{\max(\C_{j+1}) - i}{\ell} \right\rfloor - m_B \right) \ell + i  \notag \\
  &\geq
  \left( \left\lfloor \frac{\max(\C_{j}) - i}{\ell} \right\rfloor - \left\lfloor \frac{x^B_j -1 -i}{\ell} \right\rfloor + \left\lfloor \frac{x^B_{j+1} - i}{\ell} \right\rfloor - m_B \right) \ell + i  \notag \\
  &=
  \left( \left\lfloor \frac{\max(\C_{j}) - i}{\ell} \right\rfloor - y^B_j + \bbone_{z^B_j \leq i} + y^B_{j+1} - m_B \right) \ell + i  \notag \\
  &\geq
   \left( \left\lfloor \frac{\max(\C_{j}) - i}{\ell} \right\rfloor + 1 + \delta_{j0} \right) \ell + i \qquad (\text{by (2)---note that $i= \Ht_{\II_B}(z^B_{j+1})$}) \\
  &> \max(\C_j) + \delta_{j0} \ell \qquad \qquad \qquad \qquad \qquad  \text{(by \eqref{E:max(C_j)}).}
  \end{align*}
%
\end{description}

\item[$(3) \Rightarrow (2)$]
Suppose for a contradiction that $y^B_{j+1} - y^B_j - \delta_{j0} + \bbone_{z^B_j \leq \Ht_{\II_B}(z^B_{j+1})} \leq m_B$.
Our strategy here is to associate $B$ with a block $B^-$ with $m_{B^-} = 0$ which has an $\ell$-partition with an addable node with residue $j$ and lift this addable node to an $\ell$-partition lying in $B$.

By applying Lemma \ref{L:-} with $\ba = m_B \Be_{j+1}$ to $(\bl;\br)^{w_0}$ for any $\ell$-partition $\bl$ lying in $B$, we obtain an $\ell$-partition $\bm_0$ with \begin{align}
\quot_e(\core((\bm_0;\br - m_B\bone)^{w_0})) &= \beta_{\br^*_B - m_B\ell \Be_{j+1}}(\EP); \label{E:B-core} \\
\mv_e((\bm_0;\br - m_B\bone)^{w_0}) &= \mv_e((\bl;\br)^{w_0}) - m_B \bone = \mv(B) - m_B\bone. \notag
\end{align}
Let $\br^- = \br - m_B \bone$ and let $B^-$ be the block of $\HH_{\FF,q,\br^-}(|\bm_0|)$ in which $S^{\bm_0}$ lies.
Then $(\br^-)^{w_0} = \br^{w_0} - m_B\bone \in \AAbar$, and so
$$\mv^{(\br^-)^{w_0}}(B^-) = \mv_e((\bm_0;\br^-)^{w_0}) = \mv(B) - m_B\bone,$$
so that $m_{B^-} = 0$ and $\II_{B^-} = \II_B$.
Furthermore, 
$
(\br^{-})^*_{B^-} = 
\br^{*}_B - m_B \ell\, \Be_{j+1} 
$ by \eqref{E:bu} and \eqref{E:B-core}, so that
\begin{align*}
x^{B^-}_a &= x^B_a - \delta_{a,j+1}\,m_B\ell, \qquad
y^{B^-}_a = y^B_a - \delta_{a,j+1}\,m_B, \qquad
z^{B^-}_a = z^B_a
\end{align*}
for all $a \in [1,\,e]$.

We claim that there exists $\bm \in \PP^{\ell}$ lying in $B^-$ having an addable node with $(e,\br^{-})$-residue $j$.
If $\hub_j(B^-) < 0$, then any $\ell$-partition lying in $B^-$ (for example, $\bm_0$) will do.
Thus we only need to consider the case where $\hub_j(B^-) \geq 0$, i.e. $x^{B^-}_{j+1} \geq x^{B^-}_j + \delta_{j0} \ell$, or equivalently,
$y^{B^-}_{j+1} > y^{B^-}_j + \delta_{j0}$ or ($y^{B^-}_{j+1} = y^{B^-}_j + \delta_{j0}$ and $z^{B^-}_{j+1} \geq z^{B^-}_j$).
But
\begin{align*}
y^{B^-}_{j+1} = y^B_{j+1} - m_B &\leq y^B_j + \delta_{j0} - \bbone_{z^B_j \leq \Ht_{\II_B}(z^B_{j+1})} \\
&= y^{B^-}_j + \delta_{j0} - \bbone_{z^{B^-}_j \leq \Ht_{\II_{B^-}}(z^{B^-}_{j+1})}
\leq y^{B^-}_j + \delta_{j0}.
\end{align*}
As such it suffices to consider only the case where $y^{B^-}_{j+1} = y^{B^-}_j + \delta_{j0}$ and $z^{B^-}_{j+1} \geq z^{B^-}_j > \Ht_{\II_{B^-}}(z^{B^-}_{j+1})$, with the latter further implying that $\Ht_{\II_{B^-}}(z^{B^-}_j) = \Ht_{\II_{B^-}}(z^{B^-}_{j+1})$ by the definition of $\Ht_{\II_{B^-}}$.
The existence of $\bm$ lying in $B^-$ having an addable node with $(e,\br^{-})$-residue $j$ is thus guaranteed by \cite[Proposition 4.23]{LQT}.

Let $\bm^* = (\mu^{*(1)}, \dotsc, \mu^{*(e)}) \in \PP^e$ be such that 
$\bij_{\ell,e}((\bm;\br^-)^{w_0}) = (\bm^*, (\br^-)^*_{B^-})$.
Then
$$\quot_e(\UU(\beta_{(\br^{-})^{w_0}}(\bm^{w_0}))) = (\beta_{x^{B^-}_1}(\mu^{*(1)}),\dotsc, \beta_{x^{B^-}_e}(\mu^{*(e)})) =: (\B_1,\dotsc, \B_e).$$
Since $\bm$ has an addable node with $(e,\br^-)$-residue $j$, there exists $b \in \B_j$ such that $b + \delta_{j0} \ell \notin \B_{j+1}$ (where $\B_0 := \B_e$) by Proposition \ref{P:abacus}(2), and we pick $b$ to be the largest such.  
Since $b + \delta_{j0}\ell \notin \B_{j+1} = \beta_{x^{B^-}_{j+1}}(\mu^{*(j+1)})$, and $\mu^{*(j+1)}$ does not have nodes of every $(\ell,x^{B^-}_{j+1})$-residue as $\mv_e((\bm;\br^-)^{w_0})= \mv^{(\br^-)^{w_0}}(B^-)$ and $m_{B^-} = 0$, we see $b+ \delta_{j0}\ell+ k \notin \B_{j+1}$ for all $k \in \ZZ_{\geq \ell}$, which in turn shows that $b+ k \notin \B_j$ for all $k \in \ZZ_{\geq \ell}$ by the maximality of $b$.

Let $\nu = \beta^{-1}((\B_j \setminus \{b \}) \cup \{b + m_B\ell\})$.
Then $\nu$ is obtained by $\mu^{*(j)}$ by adding a rim hook of size $m_B\ell$. Let
$\bnu^* = (\mu^{*(1)}, \dotsc, \mu^{*(j-1)}, \nu, \mu^{*(j+1)}, \dotsc, \mu^{*(e)})$, and let 
$(\bnu;\bt) \in \PP^{\ell} \times \ZZ^{\ell}$ be such that $\bij_{\ell,e}(\bnu;\bt) = (\bnu^*;\br^*_B)$.
Then
\begin{gather}
\quot_e(\core_e(\bnu;\bt)) = \beta_{\br^*_B}(\EP) = \quot_e(\beta_{|\br^{w_0}|}(\core(B))); \label{E:core} \\
\mv_e(\bnu;\bt) = \mv_e((\bm;\br^-)^{w_0}) + m_B\bone = \mv^{(\br^-)^{w_0}}(B^-) + m_B \bone = \mv(B). \label{E:mv}
\end{gather}
Furthermore, from 
\begin{gather*}
\quot_e(\UU(\beta_{\bt}(\bnu))) = \bij_{\ell,e}((\bnu;\bt)) = (\bnu^*;\br^*_B) = (\bnu^*;(\br^-)^*_{B^-} + m_B \ell \Be_{j+1}), \\
\quot_e(\UU(\beta_{(\br^-)^{w_0}}(\bm^{w_0}))) = \bij_{\ell,e}((\bm;\br^-)^{w_0}) = (\bm^*;(\br^-)^*_{B^-}),
\end{gather*} 
and how $\bm^*$ and $\bnu^*$ are related, we deduce that  
$\bt = (\br^-)^{w_0} + m_B \bone = \br^{w_0}.$
Consequently, $\core_{\HH}(\bnu^{w_0^{-1}}) = \core(B)$ and $\wt_{\HH}(\bnu^{w_0^{-1}}) = \wt(B)$ by \eqref{E:core}, \eqref{E:mv} and \eqref{E:wt-mv}, and so $\bnu^{w_0^{-1}}$ lies in $B$ by Theorem \ref{T:Naka}.
Since $b+ m_B\ell \in \beta_{\fs(\B_j)}(\nu) = \beta_{x^{\smash[t]{B^-}}_j}(\nu) = \beta_{x^B_j}(\nu)$ while $$b+ m_B\ell + \delta_{j0}\ell \notin \B_{j+1}^{+m_B\ell} = (\beta_{x^{\smash[t]{B^-}}_{j+1}}(\mu^{*(j+1)}))^{+m_B\ell} 
= \beta_{x^{\smash[t]{B^-}}_{j+1} + m_B\ell}(\mu^{*(j+1)}) = \beta_{x^B_{j+1}}(\mu^{*(j+1)}),$$ we see that $\bnu$ has an addable node with $(e,\bt)$-residue $j$, and hence $\bnu^{w_0^{-1}}$ has an addable node with $(e,\br)$-residue $j$, as desired.
\end{description}
\end{proof}

\begin{rem}
When $\ell = 1$, then $\ZZ/\ell\ZZ = \{0\} = \II_B$, $m_B = \wt(B)$ for all $B$.
Furthermore, $B$ and $\cs_j \DDot{} B$ form a $[w: k]$-pair where $w = \wt(B)$ and $k = y^B_{j+1} - y^B_j - \delta_{j0}$.  Thus statement (2) of Theorem \ref{T:equiv1} reduces to saying that $B$ and $\cs_j \DDot{} B$ form a $[w: k]$-pair with $k \geq w$. This is of course consistent with what Scopes found in her work on Scopes equivalence in level 1 \cite{Scopes}.
\end{rem}

Note that Theorem \ref{T:equiv1} is a generalisation of Corollary 4.4 and Proposition 4.23 of \cite{LQT} which dealt with the case of core blocks of Ariki-Koike algebras, or equivalently, the blocks for which $m_B = 0$ (see \cite[Theorem 3.17]{LQT}).
As \cite{LQT} already provided a necessary and sufficient condition for two core blocks to be Scopes equivalent \cite[Theorem 4.24]{LQT}, we shall focus on non-core blocks---where $m_B > 0$---in what follows, and provide a necessary and sufficient condition for two such blocks to be Scopes equivalent.

We now introduce the two invariants that will be used to classify Scopes equivalence classes of non-core blocks. 
These invariants are analogues of those appearing in earlier classifications for Iwahori–Hecke algebras due to Richards \cite{R} and for core blocks in \cite{LQT}.
Our next main result shows that these are the relevant invariants.

\begin{Def} \label{D:pyramid}
Let $B$ be a block of $\HH_n$, and recall the notations introduced in Definition \ref{D:}.
\begin{enumerate}
\item For $a, b \in [1,\, e]$ with $a < b$, define
\begin{align*}
\dd_B(a,b) &:= y^B_{\sigma_B(b)} - y^B_{\sigma_B(a)} - \bbone_{\sigma_B(a) > \sigma_B(b)}  + \bbone_{z^B_{\sigma_B(a)} \leq \Ht_{\II_B}(z^B_{\sigma_B(b)})}; \\
\pi_B(a,b) &:= \min(\dd_B(a,b), m_B).
\end{align*}
We call $\{ \pi_B(a,b) \mid a, b \in [1,\,e],\, a< b \}$ the {\em pyramid numbers} of $B$.

\item If $m_B > 0$, we call $(\bz^B)^{\sigma_B}$ the {\em Scopes vector} of $B$.
\end{enumerate}
\end{Def}

\begin{rem} \hfill
\begin{enumerate}
\item When $m_B = 0$, the Scopes vector of $B$ has already been defined in \cite[Definition 4.14]{LQT}, and is used to completely classify the Scopes equivalence classes for such blocks (\cite[Theorem 4.24]{LQT}).
Perhaps counter-intuitively, the definition of a Scopes vector for a core block is slightly more complicated than that for a non-core block.

\item When $\ell = 1$ so that $m_B = \wt(B)$, our pyramid number $\pi_B(a,b)$ equals the pyramid number $\wt(B)-\pi_B(a,b)-1$ of Richards in \cite{R}, while the Scopes vector of $B$ equals $\textbf{0}$, irrespective of whether $m_B$ equals 0 or is positive.
\end{enumerate}
\end{rem}

\begin{eg}
Let $\br = (0,\dotsc, 0 , 1)\ (= \br^{w_0})$ and $\ell \geq 2$.
\begin{enumerate}
\item
Let $B$ be the block of $\HH_{e\ell}$ with $\br^*_B = (1,0,\dotsc, 0)$ and $\mv(B) = (1,\dotsc, 1)$.
Then $m_B =1$, $\by^B = (0,\dotsc, 0)$, $\bz^B = (1,0,\dotsc,0)$ and $\sigma_B = 1_{\sym{e}}$.
Thus $(\bz^B)^{\sigma_B} = (1,0,\dotsc, 0)$ and
$$
\pi_B(a,b) = \dd_B(a,b) =
\begin{cases}
0, &\text{if } a = 1; \\
1, &\text{otherwise.}
\end{cases}
$$

\item
Let $C$ be the block of $\HH_{e\ell + (e-1)(\ell-1)}$ with $\br^*_C = (\ell,1-\ell,0,\dotsc, 0)$ and $\mv(C) = (1,\dotsc, 1)$.
Then $m_C = 1$, $\by^C = (1,-1,0,\dotsc,0)$, $\bz^C = (0,1,0,\dotsc,0)$ and $\sigma_C = \cs_1 \cs_2 \dotsm \cs_{e-1}$.
Thus $(\bz^C)^{\sigma_C} = (1,0,\dotsc, 0)$ and $\pi_C(a,b) = \dd_C(a,b) = 1$ for all $a,b \in [1,\,e]$ with $a < b$.
\end{enumerate}
\end{eg}

\begin{lem} \label{L:pyramid-core}
Let $B$ be a block of $\HH_n$.  Then $0 \leq \pi_B(a,b) \leq m_B$ for all $a,b \in [1,\,e]$ with $a < b$.

In particular, if $B$ is a core block (i.e.\ $m_B = 0$), then $\pi_B(a,b) = 0$ for all $a,b \in [1,\,e]$ with $a < b$.
\end{lem}

\begin{proof}
By the definition of $\sigma_B = \sigma_{\by^B}$, for $a,b \in [1,\,e]$ with $a < b$, we have $ y^B_{\sigma_B(a)} + \bbone_{\sigma_B(a) > \sigma_B(b)} \leq y^B_{\sigma_B(b)}$, so that
 $$\dd_B(a,b) = y^B_{\sigma_B(b)} - y^B_{\sigma_B(a)} - \bbone_{\sigma_B(a) > \sigma_B(b)}  + \bbone_{z^B_{\sigma_B(a)} \leq \Ht_{\II_B}(z^B_{\sigma_B(b)})} \geq 0.$$
\end{proof}

\begin{prop} \label{P:same-residue-pyramid}
Let $B$ be a block of $\HH_n$ with $m_B > 0$. Let $j \in \ZZ/e\ZZ$, and 
suppose that $B \xrightarrow{j} C$.
Then
$(\bz^B)^{\sigma_B} = (\bz^C)^{\sigma_C}$ and $\pi_B = \pi_C$.
\end{prop}

\begin{proof}
By Theorem \ref{T:equiv1}(2),
$y^B_{j+1} - y^B_j - \delta_{j0} \geq m_B > 0$.
Since $C = \cs_j \DDot{} B$, we have by Lemma \ref{L:cs_j}
\begin{align*}
\sigma_C 
&= \overline{\cs_j} \sigma_{B}, \\
(\bz^C)^{\sigma_C} &= ((\bz^B)^{\overline{\cs_j}})^{\overline{\cs_j}\sigma_B} = (\bz^B)^{\sigma_B}; \\
(\by^C)^{\sigma_C} &= ((\by^B)^{\overline{\cs_j}} + \delta_{j0}(\Be_1 - \Be_e))^{\overline{\cs_j}\sigma_B}
= (\by^B)^{\sigma_B} + \delta_{j0}(\Be_{\sigma_B^{-1}(e)} - \Be_{\sigma_B^{-1}(1)}).
\end{align*}
Furthermore, since $\mv(B) = \mv(C)$ by Proposition \ref{P:moving-vector-Weyl-orbit}, we have $\II_B = \II_C$.
Consequently,
for $a,b \in [1,\, e]$ with $a< b$,
\begin{align*}
\dd_{C}(a,b)
&= y^{C}_{\sigma_C(b)} - y^{C}_{\sigma_{C}(a)} - \bbone_{\sigma_{C}(a) > \sigma_{C}(b)} + \bbone_{z^{C}_{\sigma_{C}(a)} \leq \Ht_{\II_C} (z^{C}_{\sigma_{C}(b)})} \\
&= y^B_{\sigma_B(b)} + \delta_{j0}(\delta_{e,\sigma_B(b)} - \delta_{1,\sigma_B(b)}) - y^B_{\sigma_B(a)} - \delta_{j0}(\delta_{e,\sigma_B(a)} - \delta_{1,\sigma_B(a)}) \\
&\qquad - \bbone_{\overline{\cs_j}\sigma_B(a) > \overline{\cs_j}\sigma_B(b)} + \bbone_{z^{B}_{\sigma_{B}(a)} \leq \Ht_{\II_B} (z^{B}_{\sigma_B(b)})} \\
&= d_B(a,b) + \delta_{j0}(\delta_{e,\sigma_B(b)} + \delta_{1,\sigma_B(a)} - \delta_{e,\sigma_B(a)} - \delta_{1,\sigma_B(b)} ) \\
&\qquad - \bbone_{\overline{\cs_j}\sigma_B(a) > \overline{\cs_j}\sigma_B(b)} + \bbone_{\sigma_B(a) > \sigma_B(b)}.
\end{align*}
When $j > 0$, then $\bbone_{\overline{\cs_j}\sigma_B(a) > \overline{\cs_j}\sigma_B(b)} = \bbone_{\sigma_B(a) > \sigma_B(b)}$ unless $\{\sigma_B(a),\sigma_B(b)\} = \{j,j+1\}$, in which case $\sigma_B(a) = j$ and $\sigma_B(b) = j+1$ since $y^B_{j+1} > y^B_j + \delta_{j0}$ and $y^B_{\sigma_B(a)} \leq y^B_{\sigma_B(b)}$.
Thus we have
$$
\dd_C(a,b) = \dd_B(a,b)- \delta_{j,\sigma_B(a)}\delta_{j+1,\sigma_B(b)}
$$
in this case.
On the other hand, when $j = 0$, then $\sigma_B(a) = 1$ or $\sigma_B(b) = e$ will imply that $\sigma_B(a) < \sigma_B(b)$ and $\overline{\cs_j}\sigma_B(a) > \overline{\cs_j}\sigma_B(b)$, while $\sigma_B(b) = 1$ or $\sigma_B(a) = e$ will imply that $\sigma_B(b) < \sigma_B(a)$ and $\overline{\cs_j}\sigma_B(b) > \overline{\cs_j}\sigma_B(a)$.
Furthermore, since $y^B_{\sigma_B(a)} \leq y^B_{\sigma_B(b)}$ and $y^B_1 = y^B_{j+1} > y^B_j + \delta_{j0} = y^B_e + 1$, we must have $\sigma_B(a) = e$ and $\sigma_B(b) = 1$ if $\{\sigma_B(a),\sigma_B(b)\} = \{1,e\}$.
From this, it is straightforward to verify that
$$
\dd_{C}(a,b) = \dd_B(a,b) - \delta_{e,\sigma_B(a)}\delta_{1,\sigma_B(b)}.
$$
Thus, for $j \in \ZZ/e\ZZ$ arbitrary, we have
$$
\dd_C(a,b) =
\begin{cases}
\dd_B(a,b) - 1, &\text{if } \sigma_B(a) \equiv_e j,\ \sigma_B(b) = j+1; \\
\dd_B(a,b), &\text{otherwise.}
\end{cases}
$$
Since $\dd_B(\sigma_B^{-1}(j),\sigma_B^{-1}(j+1)) \geq m_B+1$ by Theorem \ref{T:equiv1} (where $\sigma_B^{-1}(j)$ is to be read as $\sigma_B^{-1}(e)$ if $j =0$), we see that
$$
\pi_{C}(a,b) = \min(\dd_{C}(a,b), m_B) = \min(\dd_B(a,b), m_B) = \pi_B(a,b)
$$
for all $a,b \in [1,\,e]$ with $a < b$.
\end{proof}

Since $B \xrightarrow{j} C$ is the generating relation for Scopes equivalence, the Scopes vector and pyramid numbers are indeed invariants for non-core blocks in the same Scopes equivalence class:

\begin{cor} \label{C:Scopes}
Scopes equivalent non-core blocks of Ariki-Koike algebras have the same Scopes vector and the same pyramid numbers.
\end{cor}

\begin{Def} \label{D:Scopes-initial}
Let $B$ be a block of $\HH_n$ with $m_B > 0$.
We say that $B$ is {\em Scopes-initial} if and only if
$y^B_{j+1} - y^B_j - \delta_{j0} + \bbone_{z^B_{j} \leq \Ht_{\II_B}(z^B_{j+1})} \leq m_B$
for all $j \in \ZZ/e\ZZ$.
\end{Def}

\begin{rem} \label{R:initial}
Let $\sigma_B^{-1}(j) = a$ and $\sigma_B^{-1}(j+1) = b$.
If $a > b$, then $y^B_{j+1} = y^B_{\sigma_B(b)} \leq y^B_{\sigma_B(a)} = y^B_j$ by the definition of $\sigma_B = \sigma_{\by^B}$, so that
$$
y^B_{j+1} - y^B_j - \delta_{j0} + \bbone_{z^B_{j} \leq \Ht_{\II_B}(z^B_{j+1})} \leq 1 \leq m_B.
$$
Thus $B$ is Scopes-initial if and only if $y^B_{j+1} - y^B_j - \delta_{j0} + \bbone_{z^B_{j} \leq \Ht_{\II_B}(z^B_{j+1})} \leq m_B$ for all $j \in \ZZ/e\ZZ$ such that $\sigma^{-1}_B(j) < \sigma_B^{-1}(j+1)$.
\end{rem}

\begin{prop} \label{P:Scopes}
Let $B$ be a block of $\HH_n$ with $m_B > 0$.
There exists a sequence $B_0, \dotsc, B_k$ of blocks such that:
\begin{enumerate}
\item $B_0 = B$, and $B_k$ is Scopes-initial;
\item for each $a \in [0, k-1]$, $B_a \xrightarrow{j_a} B_{a+1}$ for some $j_a \in \ZZ/e\ZZ$.
\end{enumerate}
In particular, $B$ and $B_k$ are Scopes equivalent.
\end{prop}

\begin{proof}
We prove by induction on $n$.
If $B$ is Scopes-initial, then $k=0$ and there is nothing to prove.
If $B$ is not Scopes-initial,
then there exists $j \in \ZZ/e\ZZ$ such that
\begin{equation} \label{E:cs_j-Scopes}
y^B_{j+1} - y^B_j - \delta_{j0} + \bbone_{z^B_{j} \leq \Ht_{\II_B}(z^B_{j+1})} \geq m_B + 1,
\end{equation}
so that $B \xrightarrow{j} \cs_j \DDot{} B$, 
and
$$
\hub_j(B) = x^B_{j+1} - x^B_j - \delta_{j0} \ell
= (y^B_{j+1} - y^B_j - \delta_{j0})\ell + (z^B_{j+1} - z^B_j)
>0,
$$
since $y^B_{j+1} - y^B_j - \delta_{j0} \geq m_B > 0$ from \eqref{E:cs_j-Scopes} and $|z^B_{j+1} - z^B_j| < \ell$.
Let 
$B_1 = \cs_j \DDot{} B$.
Since $\hub_j(B) >0$, the block $B_1 = \cs_j \DDot{} B$ lies in $\HH_{n-\hub_j(B)}$ which has strictly smaller rank.  Consequently, by induction we have the desired sequence of blocks for $B_1$ and hence for $B$.
\end{proof}

\begin{prop} \label{P:unique}
Let $B$ be a block of $\HH_n$ with $m_B > 0$.  There is at most one Scopes-initial block in $\AW_e \DDot{} B$ with the same Scopes vector and the same pyramid numbers as $B$.
\end{prop}

We defer the technical combinatorial proof of Proposition \ref{P:unique} to the next subsection and first derive its consequences.

\begin{thm} \label{T:Scopes}
Let $B$ be a block of $\HH_n$ with $m_B > 0$, and let $C \in \AW_e \DDot{} B$.  Then $B$ and $C$ are Scopes equivalent if and only if they have the same Scopes vector and the same pyramid numbers,
in which case the Specht module $\Sp^{\bl}$ lying in $B$ corresponds to the Specht module $\Sp^{\bm}$ lying in $C$ if and only if $(\bl^*)^{\sigma_B} = (\bm^*)^{\sigma_C}$, where
$\bl^*,\bm^* \in \PP^e$ satisfy $\bij_{\ell,e}((\bl;\br)^{w_0}) = (\bl^*,\br^*_B)$ and $\bij_{\ell,e}((\bm;\br)^{w_0}) = (\bm^*,\br^*_C)$ respectively.
\end{thm}

\begin{proof}
  We have already seen in Corollary \ref{C:Scopes} that if $B$ and $C$ are Scopes equivalent, then they have the same Scopes vector and the same pyramid numbers.

  Suppose then that $B$ and $C$ have the same Scopes vector and the same pyramid numbers.
  Then $B$ (resp.\ $C$) is Scopes equivalent to a Scopes-initial block $B_0$ (resp.\ $C_0$) by Proposition \ref{P:Scopes}.
  Now $B_0$ and $C_0$ share the same Scopes vector and the same pyramid numbers as $B$ and $C$ by Corollary \ref{C:Scopes},
  so that $B_0 = C_0$ by Proposition \ref{P:unique}.
  Thus $B$ and $C$ lie in the same Scopes equivalence class.
  
  We now show the correspondence between the Specht modules when $B$ and $C$ are Scopes equivalent.
  Let $B_0, \dotsc, B_k$ be a sequence of blocks such that $B_0 = B$, $B_k = C$, and for all $a \in [1,\, k]$, $B_{a} = \cs_{j_a} \DDot{} B_{a-1}$ and either all partitions lying in $B_{a-1}$ have no addable nodes with $(e,\br)$-residue $j_a$ or all partitions lying in $B_{a}$ have no addable nodes with $(e,\br)$-residue $j_a$.
  Then 
  $$\sigma_C = \sigma_{\by^C} = \sigma_{\smash[t]{\cs_{j_k} \CDot{1} ( \dotsb \CDot{1} (\cs_{j_1} \CDot{1} \by^B) \dotsm )}}
  = \overline{\cs_{j_k}} \dotsm \overline{\cs_{j_1}}\sigma_{\by^B} = \overline{\cs_{j_k}} \dotsm \overline{\cs_{j_1}}\sigma_B
  $$  
  by Lemma \ref{L:cs_j}.
  When $\bl \in \PP^{\ell}$ lying in $B$ corresponds to $\bm \in \PP^{\ell}$ lying in $C$ under this chain of blocks, we have $\bm = \cs_{j_k} \DDot{r} (\cs_{j_{k-1}} \DDot{\br} (\dotsm \DDot{\br} (\cs_{j_1} \DDot{\br} \bl)\dotsm))$ by Theorem \ref{T:CR}. Thus, if $\bij_{\ell,e}((\bl;\br)^{w_0}) = (\bl^*, \br^*_B)$ and $\bij_{\ell,e}((\bm;\br)^{w_0}) = (\bm^*, \br^*_C)$, we have
  $$
  \bm^* = (\bl^*)^{\overline{\cs_{j_1}} \dotsm \overline{\cdot \cs_{j_k}}}$$
  by Lemma \ref{L:cs_j}(1).
  Consequently,
  \begin{align*}
  (\bm^*)^{\sigma_C} = ((\bl^*)^{\overline{\cs_{j_1}} \dotsm \overline{\cdot \cs_{j_k}}})^{\overline{\cs_{j_k}} \dotsm \overline{\cs_{j_1}}\sigma_B} = (\bl^*)^{\sigma_B}
  \end{align*}
  as desired.
  \end{proof}

\begin{cor} \label{C: }
Let $B$ be a block of $\HH_n$ and $C$ be a block of $\HH_m$.  Then $B$ and $C$ are Scopes equivalent if and only if $B$ and $C$ have the same moving vector, the same Scopes vector and the same pyramid numbers.  
\end{cor}

\begin{proof}
Clearly, two Scopes equivalent blocks necessarily lie in the same $\AW_e$-orbit which is classified by their moving vectors by Proposition \ref{P:moving-vector-Weyl-orbit}.

If $m_B\, (= m_C) = 0$, then since all such blocks have all pyramid numbers equal to $0$ by Lemma \ref{L:pyramid-core}, the theorem follows from \cite[Theorem 4.24]{LQT}.

On the other hand, if $m_B\, (= m_C) >0$, the corollary follows from Theorem \ref{T:Scopes}.
\end{proof}

\subsection{Proof of Proposition \ref{P:unique}}

The goal of this subsection is to show that a Scopes-initial block is uniquely determined by its Scopes vector and pyramid numbers.

We first extend Definitions \ref{D:}(3), \ref{D:pyramid} and \ref{D:Scopes-initial} to a more general setting.

\begin{Def}
Let $\bx = (x_1,\dotsc, x_e) \in \ZZ^e$,  $I \subseteq \ZZ/\ell\ZZ$ with $0 \in I$, and $m \in \ZZ^+$.
\begin{enumerate}
\item Define $\by^{\bx} = (y^{\bx}_1, \dotsc, y^{\bx}_e) \in \ZZ^e$ and $\bz^{\bx} = (z^{\bx}_1,\dotsc, z^{\bx}_e) \in (\ZZ/\ell\ZZ)^e$ by
$$
\bx := \by^{\bx} \ell + \bz^{\bx}.
$$

\item Define, for $a,b \in [1,\,e]$ with $\sigma_{\by^{\bx}}^{-1}(a) < \sigma_{\by^{\bx}}^{-1}(b)$,
\begin{align*}
\delta_{\bx}^I(a,b) &:= y^{\bx}_{b} - y^{\bx}_{a} - \bbone_{a > b} + \bbone_{z^{\bx}_a \leq \Ht_I(z^{\bx}_b)}.
\end{align*}

\item Define for $a,b \in [1,\,e]$ with $a  < b$
\begin{align*}
\dd_{\bx}^I (a,b) &:= \delta_{\bx}^I(\sigma_{\by^{\bx}}(a),\sigma_{\by^{\bx}}(b)) \\
&\ =  y^{\bx}_{\sigma_{\by^{\bx}}(b)} - y^{\bx}_{\sigma_{\by^{\bx}}(a)} - \bbone_{\sigma_{\by^{\bx}}(a) >\sigma_{\by^{\bx}}(b)} + \bbone_{z^{\bx}_{\sigma_{\by^{\bx}}(a)} \leq \Ht_I(z^{\bx}_{\sigma_{\by^{\bx}}(b)})}; \\
\pi_{\bx}^{I,m} (a,b) &:= \min(\dd_{\bx}^I(a,b),m).
\end{align*}

\item We say that $\bx$ is {\em $(I,m)$-initial} if $\delta_{\bx}^I(j, \rr_e(j)) \leq m$ for all $j \in [1,\,e]$ with  $\sigma_{\by^{\bx}}^{-1}(j) < \sigma_{\by^{\bx}}^{-1}(\rr_e(j))$.
\end{enumerate}
\end{Def}

Given $\ba = (a_1,\dotsc, a_e) \in \ZZ^e$ and $i, j \in [1,\,e]$, write $i \preceq_{\ba} j$ for $\sigma_{\ba}^{-1}(i) \leq \sigma_{\ba}^{-1}(j)$. Then $\preceq_{\ba}$ is a total order $[1,\,e]$, 
and
$$
i \preceq_{\ba} j 
\quad \Leftrightarrow \quad \sigma_{\ba}^{-1}(i) \leq \sigma_{\ba}^{-1}(j)
\quad \Leftrightarrow \quad a_i \leq a_j \text{ with equality only if } i \leq j.
$$

Let $k \in \ZZ$. Recall the left action $\CDot{k}$ of $\EW_e$ on $\ZZ^e$ as described in Subsection \ref{SS:Weyl}.

\begin{lem} \label{L:sigma}
Let $\ba = (a_1,\dotsc, a_e) \in \ZZ^e$, and
let $\ba' = (\rr_e \Be_e) \CDot{1} \ba$ where $\rr_e = (1,2,\dotsc, e) \in \sym{e}$.  Then:
\begin{enumerate}
\item $i \prec_{\ba} j \Leftrightarrow \rr_e(i) \prec_{\ba'} \rr_e(j)$;
\item $\sigma_{\ba'} = \rr_e \sigma_{\ba}$;
\end{enumerate}
\end{lem}

\begin{proof}
Note that
$\ba' = (\rr_e \Be_e) \CDot{1} \ba = (\Be_e + \ba)^{\rr_e^{-1}} = \ba^{\rr_e^{-1}} + \Be_1$ by \eqref{E:CDot},
so that $a'_{\rr_e(j)} = a_j + \delta_{je}$ for all $j \in [1,\,e]$.
\begin{enumerate}
\item
We have
\begin{align*}
i \prec_{\ba} j &\Leftrightarrow a_i < a_j\ \vee\ (a_i = a_j\ \wedge\ i < j) \\
& \Leftrightarrow a'_{\rr_e(i)} - \delta_{ie} < a'_{\rr_e(j)} - \delta_{je}\ \vee\ (a'_{\rr_e(i)} - \delta_{ie} = a'_{\rr_e(j)} - \delta_{je}\ \wedge\ i < j) \\
&\Leftrightarrow
\begin{cases}
a'_1 -1 < a'_{\rr_e(j)}, &\text{if } i = e > j; \\
a'_{\rr_e(i)} \leq a'_1-1, &\text{if } j = e > i; \\
a'_{\rr_e(i)} < a'_{\rr_e(j)}\ \vee\ (a'_{\rr_e(i)} = a'_{\rr_e(j)}\  \wedge\ i < j),
&\text{if } i, j \ne e.
\end{cases} \\
&\Leftrightarrow
a'_{\rr_e(i)} < a'_{\rr_e(j)}\ \vee\ (a'_{\rr_e(i)} = a'_{\rr_e(j)}\ \wedge\ \rr_e(i) < \rr_e(j)) \\
&\Leftrightarrow
\rr_e(i) \prec_{\ba'} \rr_e(j).
\end{align*}

\item We have
$$
\sigma_{\ba}(1) \prec_{\ba} \sigma_{\ba}(2) \prec_{\ba} \dotsb \prec_{\ba} \sigma_{\ba}(e)$$ so that
$$
\rr_e\sigma_{\ba}(1) \prec_{\ba'} \rr_e\sigma_{\ba}(2) \prec_{\ba'} \dotsb \prec_{\ba'} \rr_e\sigma_{\ba}(e)
$$
by part (1).  Thus 
$$\sigma_{\ba'}^{-1} \rr_e\sigma_{\ba}(1) < \sigma_{\ba'}^{-1} \rr_e\sigma_{\ba}(2) < \dotsb < \sigma_{\ba'}^{-1} \rr_e\sigma_{\ba}(e).$$
Consequently, $\sigma_{\ba'}^{-1} \rr_e\sigma_{\ba} \in \sym{e}$ preserves the natural order of $\{ 1,2,\dotsc, e \}$, so that $\sigma_{\ba'}^{-1} \rr_e\sigma_{\ba} = 1_{\sym{e}}$, as desired.
\end{enumerate}
\end{proof}

\begin{prop} \label{P:invariant}
Let $\bx \in \ZZ^e$ and $I \subseteq \ZZ/\ell\ZZ$ with $0 \in I$.
Let $\bx' = \rr_e\Be_{e}\CDot{\ell} \bx$.
Then $\dd^I_{\bx'}(a,b) = \dd^I_{\bx}(a,b)$ for all $a,b \in [1,\,e]$ with $a< b$.
\end{prop}

\begin{proof}
Since $\rr_e\Be_{e}\CDot{\ell} \bx = \bx' = \by^{\bx'}\ell + \bz^{\bx'}$, we have
\begin{align*}
\by^{\bx'} &= \rr_e\Be_{e} \CDot{1} \by^{\bx} = (\Be_e + \by^{\bx})^{\rr_e^{-1}}, \\
\bz^{\bx'} &= \rr_e\Be_{e} \CDot{0} \bz^{\bx} = (\bz^{\bx})^{\rr_e^{-1}}
\end{align*}
by \eqref{E:CDot}.
Consequently, $\sigma_{\by^{\bx'}} = \rr_e \sigma_{\by^{\bx}}$ by Lemma \ref{L:sigma}(2), and
so
$
(\by^{\bx'})^{\sigma_{\by^{\bx\smash[t]{'}}}} = ( (\Be_e + \by^{\bx})^{\rr_e^{-1}} )^{\rr_e\sigma_{\by^{\bx}}} = (\Be_e + \by^{\bx})^{\sigma_{\by^{\bx}}},
$
i.e.\
$$
y^{\bx'}_{\sigma_{\by^{\bx\smash[t]{'}}}(a)} = y^{\bx}_{\sigma_{\by^{\bx}}(a)} + \delta_{\sigma_{\by^{\bx}}(a),e}
$$
for all $a \in [1,\,e]$.
Furthermore,
$(\bz^{\bx'})^{\sigma_{\by^{\bx\smash[t]{'}}}} = ((\bz^{\bx})^{\rr_e^{-1}})^{\rr_e\sigma_{\by^{\bx}}} = (\bz^{\bx})^{\sigma_{\by^{\bx}}}$, i.e.
$$
z^{\bx'}_{\sigma_{\by^{\bx\smash[t]{'}}}(a)} = z^{\bx}_{\sigma_{\by^{\bx}}(a)}
$$
for all $a \in [1,\,e]$.
Thus, for $a,b \in [1,\,e]$ with $a< b$, we have
\begin{align*}
\dd_{\bx'}^I(a,b) &= \delta_{\bx'}^I(\sigma_{\by^{\bx\smash[t]{'}}}(a), \sigma_{\by^{\bx\smash[t]{'}}}(b)) \\
&= y^{\bx'}_{\sigma_{{\by^{\bx\smash[t]{'}}}}(b)} - y^{\bx'}_{\sigma_{{\by^{\bx\smash[t]{'}}}}(a)} - \bbone_{\sigma_{\by^{\bx\smash[t]{'}}}(a) > \sigma_{\by^{\bx\smash[t]{'}}}(b)} +
\bbone_{z^{\bx'}_{\sigma_{\by^{\bx\smash[t]{'}}}(a)} \leq \Ht_I(z^{\bx'}_{\sigma_{\by^{\bx\smash[t]{'}}}(b)})} \\
&= y^{\bx}_{\sigma_{{\by^{\bx}}}(b)} - y^{\bx}_{\sigma_{{\by^{\bx}}}(a)} + (\delta_{\sigma_{\by^{\bx}}(b),e} - \delta_{\sigma_{\by^{\bx}}(a),e} - \bbone_{\rr_e\sigma_{\by^{\bx}}(a) > \rr_e\sigma_{\by^{\bx}}(b)}) +
\bbone_{z^{\bx}_{\sigma_{\by^{\bx}}(a)} \leq \Ht_I(z^{\bx}_{\sigma_{\by^{\bx}}(b)})}.
\end{align*}
Note that, when $a \ne b$,
\begin{align*}
\bbone_{\rr_e\sigma_{\by^{\bx}}(a) > \rr_e\sigma_{\by^{\bx}}(b)}
&=
\begin{cases}
0 &\text{if } \sigma_{\by^{\bx}}(a) = e \\
1 &\text{if } \sigma_{\by^{\bx}}(b) = e \\
\bbone_{\sigma_{\by^{\bx}}(a) > \sigma_{\by^{\bx}}(b)} &\text{if } \sigma_{\by^{\bx}}(a), \sigma_{\by^{\bx}}(b) \ne e
\end{cases} \\
&= \bbone_{\sigma_{\by^{\bx}}(a) > \sigma_{\by^{\bx}}(b)} - \delta_{\sigma_{\by^{\bx}}(a),e} + \delta_{\sigma_{\by^{\bx}}(b),e}.
\end{align*}
Thus
$$
\dd_{\bx'}^I(a,b) = y^{\bx}_{\sigma_{{\by^{\bx}}}(b)} - y^{\bx}_{\sigma_{{\by^{\bx}}}(a)} - \bbone_{\sigma_{\by^{\bx}}(a) > \sigma_{\by^{\bx}}(b)} +
\bbone_{z^{\bx}_{\sigma_{\by^{\bx}}(a)} \leq \Ht_I(z^{\bx}_{\sigma_{\by^{\bx}}(b)})} = \dd_{\bx}^I(a,b)
$$
as desired.
\end{proof}

\begin{lem} \label{L:delta-basic}
Let $\bx \in \ZZ^e$ and $I \subseteq \ZZ/\ell\ZZ$ with $0 \in I$.
Let $a,b,c \in [1,\,e]$.
\begin{enumerate}
\item If $a \prec_{\by^\bx} b$, then $\delta_{\bx}^I(a,b) \geq 0$, with equality if and only if $y^{\bx}_{b} = y^{\bx}_{a} + \bbone_{a > b}$ and $z^{\bx}_{a} > \Ht_I(z^{\bx}_{b})$;

\item If $a \prec_{\by^\bx} b \prec_{\by^\bx} c$, then
$$\delta_{\bx}^I(a,c) = \delta_{\bx}^I(a,b) + \delta_{\bx}^I(b,c) - \ff^{\bz^{\bx},I}_{a,b,c} + \bbone_{a < c < b} + \bbone_{c < b < a} + \bbone_{b < a < c}.
$$
\end{enumerate}
Here, and hereafter, 
\begin{align*}
\ff^{\bz^{\bx},I}_{a,b,c}
:= \begin{cases}
1, &\text{if } z^{\bx}_{a} \leq \Ht_I(z^{\bx}_b) \leq z^{\bx}_{b} \leq \Ht_I(z^{\bx}_{c}); \\
1, &\text{if } z^{\bx}_{b} \leq \Ht_I(z^{\bx}_{c}) < z^{\bx}_{a}; \\
1, &\text{if } \Ht_I(z^{\bx}_{c}) < z^{\bx}_{a} \leq \Ht_I(z^{\bx}_{b});\\
0, &\text{otherwise}.
\end{cases}
\end{align*}
\end{lem}

\begin{proof}  \hfill
\begin{enumerate}
\item When $a \prec_{\by^{\bx}} b$, by definition of $\prec_{\by^{\bx}}$, we have $y^{\bx}_a \leq y^{\bx}_b$ with equality only if $a < b$.
    Consequently, $y^{\bx}_{b} \geq y^{\bx}_{a} + \bbone_{a > b}$.
    Thus,
    $$
    \delta_{\bx}^I(a,b) = y^{\bx}_{b} - y^{\bx}_{a} - \bbone_{a > b} + \bbone_{z^{\bx}_a \leq \Ht_I(z^{\bx}_b)} \geq 0,
    $$
    with equality if and only if $y^{\bx}_{b} - y^{\bx}_{a} - \bbone_{a > b} = 0$ and $\bbone_{z^{\bx}_a \leq \Ht_I(z^{\bx}_b)} = 0$.

\item
Firstly, it is straightforward to verify that for distinct $a,b,c \in [1,e]$,
\begin{align*}
\bbone_{a > b} + \bbone_{b > c} - \bbone_{a > c} &= \bbone_{a < c < b} + \bbone_{c < b < a} + \bbone_{b < a < c}; \\
\bbone_{z^{\bx}_a \leq \Ht_I(z^{\bx}_b)} + \bbone_{z^{\bx}_b \leq \Ht_I(z^{\bx}_c)} - \bbone_{z^{\bx}_a \leq \Ht_I(z^{\bx}_c)} & = \ff^{\bz^{\bx},I}_{a,b,c}.
\end{align*}
Thus,
\begin{align*}
\delta_{\bx}^I(a,b) + \delta_{\bx}^I(b,c) &=
(y^{\bx}_{b} - y^{\bx}_{a} - \bbone_{a > b} + \bbone_{z^{\bx}_a \leq \Ht_I(z^{\bx}_b)})
+
(y^{\bx}_{c} - y^{\bx}_{b} - \bbone_{b > c} + \bbone_{z^{\bx}_b \leq \Ht_I(z^{\bx}_c)}) \\
&= \delta_{\bx}^I(a,c) + \ff^{\bz^{\bx},I}_{a,b,c} - (\bbone_{a < c < b} + \bbone_{c < b < a} + \bbone_{b < a < c}). 
\end{align*}
\end{enumerate}
\end{proof}

To easily describe the condition $a < c < b$ or $c < b < a$ or $b < a < c$, we define the following, which should be thought of as the wrapped-around integer interval between $a$ and $b$ in $[1,e]$.

\begin{Def}
For $a,b \in [1,\,e]$ with $a \ne b$, with $b = \rr_e^{k}(a)$ where $k \in [1,\, e-1]$, define
\begin{align*}
\ft{a}{b} 
& :=
\{ \rr_e(a), \rr_e^2(a), \dotsc, \rr_e^{k-1}(a) \} \\
&\ =\begin{cases}
[a+1,\, b-1], &\text{if } a < b; \\
[a+1,\,e] \cup [1,\,b-1], &\text{if } a > b.
\end{cases}
\end{align*}
\end{Def}

\begin{rem}
We do not define $\ft{a}{b}$ for $a = b$.  Thus whenever an expression $\ft{a}{b}$ appears, we shall assume that $a\ne b$ (for such expression to be defined).
\end{rem}

The following lemma may be easily verified.

\begin{lem} \label{L:ee}
Let $a,b,c \in [1,\,e]$.  The following statements are equivalent:
\begin{enumerate}
\item $c \in \ft{a}{b}$;
\item $\rr_e(c) \in \ft{\rr_e(a)}{\rr_e(b)}$;
\item $a < c < b$ or $b < a < c$ or $c < b < a$;
\item $|\ft{a}{c}| < |\ft{a}{b}|$;
\item $|\ft{c}{b}| < |\ft{a}{b}|$;
\item $\ft{a}{b} = \ft{a}{c} \cup \{c\} \cup \ft{c}{b}$ (disjoint union).
\end{enumerate}
In particular,
\begin{align}
c \in \ft{a}{b} \Leftrightarrow a \in \ft{b}{c} &\Leftrightarrow b \in \ft{c}{a} \notag \\
&\Leftrightarrow a, b, c \text{ distinct and } c \notin \ft{b}{a} \notag \\
&\Leftrightarrow a, b, c \text{ distinct and } a \notin \ft{c}{b} \notag \\
&\Leftrightarrow a, b, c \text{ distinct and } b \notin \ft{a}{c}. \label{E:ee}
\end{align}
\end{lem}

\begin{cor} \label{C:impt}
Let $a_1,\dotsc, a_k \in [1,\,e]$, and suppose that $|\ft{a_1}{a_k}| < |\ft{a_2}{a_k}| < \dotsb < |\ft{a_{k-1}}{a_k}|$.
Then, for $r,s,t \in [1,\, k]$, we have
$a_s \in \ft{a_r}{a_t}$ if and only if $r,s,t$ are distinct and ($r > s > t$ or $s > t > r$ or $t > r > s$).
\end{cor}

\begin{proof}
Suppose first that $r > s > t$.  If $r = k$, then $|\ft{a_t}{a_k}| < |\ft{a_s}{a_k}|$, so that $a_t \in \ft{a_s}{a_k}$ by Lemma \ref{L:ee}(1,5) and hence $a_s \in \ft{a_k}{a_t} = \ft{a_r}{a_t}$.
On the other hand, if $r < k$, then
$|\ft{a_t}{a_k}| < |\ft{a_s}{a_k}| < |\ft{a_r}{a_k}|$, so that by Lemma \ref{L:ee}(5,6),
\begin{alignat*}{2}
  \ft{a_r}{a_k}
  &= \ft{a_r}{a_s} \cup \ft{a_s}{a_k} \cup \{a_s\} & \quad & (\text{since }|\ft{a_s}{a_k}| < |\ft{a_r}{a_k}|) \\
  &= \ft{a_r}{a_s} \cup \ft{a_s}{a_t} \cup \ft{a_t}{a_k} \cup \{ a_s,a_t\} & \quad & (\text{since }|\ft{a_t}{a_k}| < |\ft{a_s}{a_k}|)\\
  &= \ft{a_r}{a_t} \cup \ft{a_t}{a_k} \cup \{a_t\} & \quad & (\text{since }|\ft{a_t}{a_k}| < |\ft{a_r}{a_k}|).
\end{alignat*}
Since all the above unions are disjoint, we see that $a_s \in \ft{a_r}{a_k} \setminus (\ft{a_t}{a_k} \cup \{ a_t\}) = \ft{a_r}{a_t}$.

If $s > t > r$, then above argument shows $a_t \in \ft{a_s}{a_r}$ so that $a_s \in \ft{a_r}{a_t}$ by \eqref{E:ee}.  Similar argument also holds if $t > r > s$.

Conversely, if $a_s \in \ft{a_r}{a_t}$, then $a_r, a_s, a_t$ must be distinct by Lemma \ref{L:ee}(1,3), and so $r,s,t$ are distinct.
Now if ($r > s > t$ or $s > t > r$ or $t > r > s$) is false while $r,s,t$ are distinct, we must have $r > t > s$ or $t > s > r$ or $s > r > t$, in which case the above argument shows that $a_s \in \ft{a_t}{a_r}$, so that $a_s \notin \ft{a_r}{a_t}$ by \eqref{E:ee}, a contradiction.
This completes the proof.
\end{proof}

\begin{cor} \label{C:ee} 
Let $\sigma \in \sym{e}$.
Then $\sigma$ is completely determined by $\sigma(1)$ and
$$
T_{\sigma} := \{(a,b,c) \in [1,\,e]^3 \mid \sigma(c) \in \ft{\sigma(a)}{\sigma(b)}\}.
$$
\end{cor}

\begin{proof} 
Let $\tau = \rr_e^{-\sigma(1)+1} \sigma$.  Then $\tau(1) = 1$, and by Lemma \ref{L:ee}(1,2),
$$
\{ (a,b,c) \in [1,\,e]^3 \mid \tau(c) \in \ft{\tau(a)}{\tau(b)} \}
= \{ (a,b,c) \in [1,\,e]^3 \mid \sigma(c) \in \ft{\sigma(a)}{\sigma(b)} \} = T_{\sigma}.
$$
Since $\ft{\tau(1)}{\tau(a)} = [2,\, \tau(a)-1]$ for all $a \in [2,\, e]$, we have $\tau(b) \in \ft{\tau(1)}{\tau(a)}$ if and only if $\tau(b) < \tau(a)$ for any distinct $a, b \in [2,\,e]$.
In particular, we know how $\tau(2), \dotsc, \tau(e)$ are ordered (with respect to the natural order on $[1,\,e]$) from the knowledge of $T_{\sigma}$.
The least of these must be $2$, the next one must be $3$, and so on, and the largest must be $e$.
This shows that $\tau$ is completely determined by $T_{\sigma}$, and hence so is $\sigma = \rr_e^{\sigma(1)-1} \tau$ when $\sigma(1)$ is known.
\end{proof}

\begin{prop} \label{P:unique-x}
Let $\bx \in \ZZ^e$, and let $I \subseteq \ZZ/\ell\ZZ$ with $0 \in I$.
Then $\bx$ is completely determined by $|\bx|$, $\dd_{\bx}^I$, $(\bz^{\bx})^{\sigma_{\by^{\bx}}}$ and $T_{\sigma_{\by^{\bx}}} := \{ (a,b,c) \in [1,\,e]^3 \mid \sigma_{\by^{\bx}}(c) \in \ft{\sigma_{\by^{\bx}}(a)}{\sigma_{\by^{\bx}}(b)} \}$.
\end{prop}

\begin{proof}
Let $r := \sigma_{\by^{\bx}}(1) -1 $ and $s := y^{\bx}_{\sigma_{\by^{\bx}}(1)}$, and
let $\bx' := (\rr_e\Be_e)^{-s e - r} \CDot{\ell} \bx$.
Then $\dd_{\bx'}^I = \dd_{\bx}^I$ by Proposition \ref{P:invariant}.
Furthermore,
\begin{align*}
\by^{\bx'} &= (\rr_e\Be_e)^{-se-r} \CDot{1} \by^{\bx} = (\by^{\bx}+ \bv_e(-se-r))^{\rr_e^r}, \\
\bz^{\bx'} &= (\rr_e\Be_e)^{-se-r} \CDot{0} \bz^{\bx} = (\bz^{\bx})^{\rr_e^r}
\end{align*}
by Lemma \ref{L:AA} and \eqref{E:CDot},
so that
$\sigma_{\by^{\smash{\bx'}}} = \rr_e^{-r} \sigma_{\by^{\bx}}$ by Lemma \ref{L:sigma}(2).
In particular,
$$
(\bz^{\bx'})^{\sigma_{\by^{\bx\smash{'}}}} = ((\bz^{\bx})^{\rr_e^r})^{\rr_e^{-r} \sigma_{\by^{\bx}}} = (\bz^{\bx})^{\sigma_{\by^{\bx}}},
$$
and
\begin{align*}
(a,b,c) \in T_{\sigma_{\by^{\bx}}}
&\Leftrightarrow \sigma_{\by^{\bx}}(c) \in \ft{\sigma_{\by^{\bx}}(a)}{\sigma_{\by^{\bx}}(b)} \\
&\Leftrightarrow \rr_e^{-r}\sigma_{\by^{\bx}}(c) \in \ft{\rr_e^{-r}\sigma_{\by^{\bx}}(a)}{\rr_e^{-r}\sigma_{\by^{\bx}}(b)} \qquad \text{by Lemma \ref{L:ee}(1,2)} \\
&\Leftrightarrow \sigma_{\by^{\smash{\bx'}}}(c) \in \ft{\sigma_{\by^{\smash{\bx'}}}(a)}{\sigma_{\by^{\smash{\bx'}}}(b)} \\
&\Leftrightarrow (a,b,c) \in T_{\sigma_{\by^{\bx\smash{'}}}}.
\end{align*}
Now $$\sigma_{\by^{\bx\smash{'}}}(1) = \rr_e^{-r} \sigma_{\by^{\bx}}(1) = \rr_e^{-r}(r+1) = 1,$$ so that
$\sigma_{\by^{\bx\smash{'}}}$ is completely determined by $T_{\sigma_{\by^{\bx\smash{'}}}} = T_{\sigma_{\by^{\bx}}}$ by Corollary \ref{C:ee}.
Furthermore,
\begin{align*}
(\by^{\bx'})^{\sigma_{\by^{\bx\smash{'}}}}
&= ((\by^{\bx} + \bv_e(-se-r))^{\rr_e^r}) ^{\rr_e^{-r}\sigma_{\by^{\bx}}} \\
&= (\by^{\bx} + \bv_e(-se-r))^{\sigma_{\by^{\bx}}} \\
&= (\by^{\bx})^{\sigma_{\by^{\bx}}} + (-s-1,\dotsc, -s-1, \underbrace{-s,\dotsc, -s}_{e-r \text{ times}})^{\sigma_{\by^{\bx}}}.
\end{align*}
Since $\sigma_{\by^{\bx}}(1) = r+1$ and the $(r+1)$-th entry of $(-s-1,\dotsc, -s-1, \underbrace{-s,\dotsc, -s}_{e-r \text{ times}})$ is $-s$, we have
$$
y^{\bx'}_{\sigma_{\by^{\bx\smash{'}}}(1)} = y^{\bx}_{\sigma_{\by^{\bx}}(1)} + (-s) = 0.
$$
Consequently,
\begin{align*}
y^{\bx'}_{\sigma_{\by^{\bx\smash{'}}}(i)} &= \dd_{\bx'}^I(1,i) +
y^{\bx'}_{\sigma_{\by^{\bx\smash{'}}}(1)} +
\bbone_{\sigma_{\by^{\bx\smash{'}}}(1) > \sigma_{\by^{\bx\smash{'}}}(i)} - \bbone_{z^{\bx'}_{\sigma_{\by^{\bx\smash{'}}}(1)} \leq \Ht_I(z^{\bx'}_{\sigma_{\by^{\bx\smash{'}}}(i)})}\\
&=
\dd_{\bx}^I(1,i) - \bbone_{z^{\bx}_{\sigma_{\by^{\bx}}(1)} \leq \Ht_I(z^{\bx}_{\sigma_{\by^{\bx}}(i)})}
\end{align*}
since $\dd_{\bx}^I = \dd_{\bx'}^I$, $y^{\bx'}_{\sigma_{\by^{\bx\smash{'}}}(1)} = 0$, $\sigma_{\by^{\bx\smash{'}}}(1) = 1$ and
$(\bz^{\bx'})^{\sigma_{\by^{\bx\smash{'}}}} = (\bz^{\bx})^{\sigma_{\by^{\bx}}}$.
Thus $(\by^{\bx'})^{\sigma_{\by^{\bx\smash{'}}}}$ is completely determined by $\dd_{\bx}^I$ and $(\bz^{\bx})^{\sigma_{\by^{\bx}}}$, and hence
$$
\bx' = \by^{\bx'}\ell + \bz^{\bx'} = \left( (\by^{\bx'})^{\sigma_{\by^{\bx\smash{'}}}} \ell + (\bz^{\bx'})^{\sigma_{\by^{\bx\smash{'}}}}\right)^{\sigma_{\by^{\bx\smash{'}}}^{-1}}
$$
is completely determined by $\dd_{\bx}^I$, $(\bz^{\bx})^{\sigma_{\by^{\bx}}} = (\bz^{\bx'})^{\sigma_{\by^{\bx\smash{'}}}}$ and $T_{\sigma_{\by^{\bx}}}$, as $\sigma_{\by^{\bx\smash{'}}}$ is completely determined by $T_{\sigma_{\by^{\bx\smash{'}}}}$ by Corollary \ref{C:ee} and $T_{\sigma_{\by^{\bx\smash{'}}}} = T_{\sigma_{\by^{\bx}}}$ as shown above.

It remains to show that $\bx$ is determined by $|\bx|$ and $\bx'$.
We have
$$
\bx' = (\rr_e\Be_e)^{-se - r} \CDot{\ell} \bx = \rr_e^{-r} \bv_e(-se-r) \CDot{\ell} \bx = (\bx + \ell \bv_e(-se-r))^{\rr_e^r},
$$
so that
$$
|\bx'| = |(\bx + \ell \bv_e(-se-r))^{\rr_e^r}| = |\bx| + \ell|\bv_e(-se-r)| = |\bx| + \ell(-se -r).
$$
Thus $s$ and $r$ are the unique quotient and remainder of $\frac{|\bx| - |\bx'|}{\ell}$ when divided by $e$ respectively, and are therefore completely determined by $|\bx|$ and $\bx'$.
Since $\bx = (\rr_e\Be_e)^{se+r} \CDot{\ell} \bx'$,
the desired result now follows.
\end{proof}

The next proposition establishes some important properties of 
$(I,m)$-initial $\bx$ that are crucial for the theorem that follows.

\begin{prop} \label{P:initial}
Let $\bx = (x_1,\dotsc, x_e) \in \ZZ^e$, $I \subseteq \ZZ/\ell\ZZ$ with $0 \in I$, and $m \in \ZZ^+$.
Suppose that $\bx$ is $(I,m)$-initial.
Let 
$a,b, c\in [1,\,e]$.

\begin{enumerate}
\item
If $a < b$ and $\sigma_{\by^\bx}(j) \notin \ft{\sigma_{\by^\bx}(a)}{\sigma_{\by^\bx}(b)}$ for all $j \in [a+1,\,b-1]$,
then
    $$
    \dd_{\bx}^I
    (a,b) \leq m.
    $$

\item
If $a < b < c$ and $\sigma_{\by^\bx}(j) \notin \ft{\sigma_{\by^\bx}(a)}{\sigma_{\by^\bx}(b)}$ for all $j \in [a+1,\, c-1]$, 
then the following
statements are equivalent:
\begin{enumerate}
\item  $\sigma_{\by^\bx}(c) \in \ft{\sigma_{\by^\bx}(a)}{\sigma_{\by^\bx}(b)}$.
\item $\pi_{\bx}^{I,m}(a,c) - \dd_{\bx}^I(a,b) - \dd_{\bx}^I(b,c) + \ff^{(\bz^{\bx})^{\sigma_{\by^{\bx}}},I}_{a,b,c} = 1$.
\item $\pi_{\bx}^{I,m}(a,c) - \dd_{\bx}^I(a,b) - \dd_{\bx}^I(b,c) + \ff^{(\bz^{\bx})^{\sigma_{\by^{\bx}}},I}_{a,b,c} > 0$.
\end{enumerate}
\end{enumerate}
\end{prop}

\begin{proof}
We adopt the following shorthands in this proof: $\by := \by^{\bx}$, $\sigma := \sigma_{\by^\bx}$, $\bz := (\bz^{\bx})^{\sigma_{\by^{\bx}}}$, $\dd := \dd_{\bx}^I$, $\pi := \pi_{\bx}^{I,m}$.
\begin{enumerate}
\item
Let $k \in [b,\,e]$ be such that $|\ft{\sigma(a)}{\sigma(k)}| = \min\{ |\ft{\sigma(a)}{\sigma(k')}| \mid  k' \in [b,\,e] \}$.
Then $k \geq b > a$, and $|\ft{\sigma(a)}{\sigma(k)}| \leq |\ft{\sigma(a)}{\sigma(b)}|$.
Let $k^- := \sigma^{-1}\rr_e^{-1}\sigma(k)$.
Then $\rr_e\sigma(k^-) = \sigma(k)$.
If $k^- > a$, then 
$$|\ft{\sigma(a)}{\sigma(k^-)}| < |\ft{\sigma(a)}{\rr_e\sigma(k^-)}| = |\ft{\sigma(a)}{\sigma(k)}| \leq |\ft{\sigma(a)}{\sigma(b)},$$
contradicting our choice of $k$ if $k^- \geq b$ and our assumption that for all $j \in [a+1,\, b-1]$, $\sigma(j) \notin \ft{\sigma(a)}{\sigma(b)}$, or equivalently
$|\ft{\sigma(a)}{\sigma(j)}| > |\ft{\sigma(a)}{\sigma(b)}|$ by \eqref{E:ee} and Lemma \ref{L:ee}(1,4), if $a< k^- < b$.
Thus $k^- \leq a < b \leq k$, and
$$
\dd(k^-,k) = \delta_{\bx}^I(\sigma(k^-), \sigma(k)) = \delta_{\bx}^I(\sigma(k^-), \rr_e\sigma(k^-)) \leq m$$
since $\bx$ is $(I,m)$-initial.
We now consider four separate cases.
\begin{description}
\item[Case 1. $k^- = a$ and $k=b$] We have $\dd(a,b) = \dd(k^-,k)  \leq m$.

\item[Case 2. $k^- = a$ and $k > b$]
By Lemma \ref{L:delta-basic}(2),
\begin{align*}
\dd(a, b) &= \dd(k^-,b) = \dd(k^-, k) - \dd(b,k) - \bbone_{\sigma(k) \in \ft{\sigma(k^-)}{\sigma(b)}} + \ff^{\bz,I}_{k^-,b,k} 
 \leq m,
\end{align*}
since $\dd(b,k) \geq 0$ by Lemma \ref{L:delta-basic}(1), $\sigma(k) = \rr_e\sigma(k^-) \in \ft{\sigma(k^-)}{\sigma(b)}$ and $\ff^{\bz,I}_{k^-,b,k} \leq 1$.

\item[Case 3. $k^- < a$ and $k= b$]
We have
\begin{align*}
\dd(a, b) = \dd(a,k) &= \dd(k^-, k) - \dd(k^-,a) - \bbone_{\sigma(k) \in \ft{\sigma(k^-)}{\sigma(a)}} + \ff^{\bz,I}_{k^-,a,k} 
\leq m,
\end{align*}
similar to case 2.

\item[Case 4. $k^- < a$ and $k > b$.]
Since $|\ft{\sigma(a)}{\sigma(k)}| < |\ft{\sigma(a)}{\sigma(b)}|$ by choice of $k$,
we have $\sigma(k) \in \ft{\sigma(a)}{\sigma(b)}$ by Lemma \ref{L:ee}(1,4), so that
\begin{align*}
\dd(a, b) &= \dd(a, k) - \dd(b,k) - \bbone_{\sigma(k) \in \ft{\sigma(a)}{\sigma(b)}} + \ff^{\bz,I}_{a,b,k} \\
&\leq \dd(a, k) \\
&= \dd(k^-,k) - \dd(k^-,a) - \bbone_{\sigma(k) \in \ft{\sigma(k^-)}{\sigma(a)}} + \ff^{\bz,I}_{k^-,a,k} \leq m,
\end{align*}
as before.
\end{description}

\item
If $\sigma(c) \in \ft{\sigma(a)}{\sigma(b)}$, then
$\ft{\sigma(a)}{\sigma(c)} \subsetneq \ft{\sigma(a)}{\sigma(b)}$ by Lemma \ref{L:ee}(1,6).
Since, by assumption, for each $j \in [a+1,\, c-1]$, $\sigma(j)\notin \ft{\sigma(a)}{\sigma(b)}$, we thus have
$\sigma(j) \notin \ft{\sigma(a)}{\sigma(c)}$, so that
$\dd(a,c) \leq m$ by part (1).
Thus
\begin{align*}
\pi(a,c) - \dd(a,b) - \dd(b,c) + \ff^{\bz,I}_{a,b,c} &= \dd(a,c) - \dd(a,b) - \dd(b,c) + \ff^{\bz,I}_{a,b,c} \\
&= \bbone_{\sigma(c) \in \ft{\sigma(a)}{\sigma(b)}} = 1
\end{align*}
by Lemma \ref{L:delta-basic}(2).

On the other hand, if $\sigma(c) \notin \ft{\sigma(a)}{\sigma(b)}$,
then
\begin{align*}
\pi(a,c) - \dd(a,b) - \dd(b,c) + \ff^{\bz,I}_{a,b,c} &\leq \dd(a,c) - \dd(a,b) - \dd(b,c) + \ff^{\bz,I}_{a,b,c} \\
&= \bbone_{\sigma(c) \in \ft{\sigma(a)}{\sigma(b)}} = 0.
\end{align*}
Part (2) thus follows.
\end{enumerate}
\end{proof}

\begin{thm} \label{T:unique-initial}
Let $I \subseteq \ZZ/\ell\ZZ$ with $0 \in I$, and $m \in \ZZ^+$.
Suppose that $\bx= (x_1,\dotsc, x_e) \in \ZZ^e$ is $(I,m)$-initial.
Then $\bx$ is completely determined by $\pi_{\bx}^{I,m}$, $(\bz^{\bx})^{\sigma_{\by^{\bx}}}$, and $|\bx|$. \end{thm}

\begin{proof}
We adopt the following shorthand in this proof: $\sigma := \sigma_{\by^{\bx}}$, $\dd := \dd_{\bx}^I$,
$\pi := \pi_{\bx}^{I,m}$, $\bz := (\bz^{\bx})^{\sigma_{\by^{\bx}}}$.

By Proposition \ref{P:unique-x}, $\bx$ is completely determined by $|\bx|$, $\dd$, $\bz$ and $T_{\sigma}$.
As such, it suffices to show that $T_{\sigma}$ and $\dd$ are completely determined by $\pi$ and $\bz$ when $\bx$ is $(I,m)$-initial, and we do this by induction with the help of Proposition \ref{P:initial}.

For $k \in [2,\,e]$, define $\dd_k : \{ (a,b) \in [1,\,k]^2 \mid a < b \} \to \mathbb{Z}$ by $\dd_k(a,b) = \dd(a,b)$, and let $$A_k := \{ (a,b) \in [1,\,k-1]^2 \mid \sigma(k) \in \ft{\sigma(a)}{\sigma(b)} \}.$$
We prove by induction on $k$ that $\dd_k$ and $A_k$ are completely determined by $\pi$ and $\bz$.
For $k = 2$, we have $A_2 = \emptyset$, and $\dd_2(1,2) = \dd(1,2) \leq m$ by Proposition \ref{P:initial}(1), so that $\dd_2(1,2) = \pi(1,2)$.

For $k > 2$,
we have $\dd_k(a,b) = \dd_{k-1}(a,b)$ for all $a,b \in [1,\, k-1]$ with $a < b$.
Since, for $a,b \in [1,\, k-2]$,
$$
\sigma(k-1) \in \ft{\sigma(a)}{\sigma(b)}\ \Leftrightarrow\ \sigma(a) \in \ft{\sigma(b)}{\sigma(k-1)} \ \Leftrightarrow\ |\ft{\sigma(a)}{\sigma(k-1)}| <  |\ft{\sigma(b)}{\sigma(k-1)}| 
$$
by \eqref{E:ee} and Lemma \ref{L:ee}(1,5),
we can determine, from the knowledge of $A_{k-1}$, the integers
$a_1, \dotsc, a_{k-2} \in [1,\,k-2]$ satisfying
$$
|\ft{\sigma(a_1)}{\sigma(k-1)}| < |\ft{\sigma(a_2)}{\sigma(k-1)}| < \dotsb < |\ft{\sigma(a_{k-2})}{\sigma(k-1)}|.
$$
For convenience, let $a_{0} = a_{k-1} = k-1$. 

We now prove by induction on $i \in [0,\, k-2]$ that we can determine if $\sigma(k) \in \ft{\sigma(a_i)}{\sigma(k-1)}$ and the value of $\dd_k(a_i,k)$ from $\pi$ and $\bz$.

For $i = 0$, $\ft{\sigma(a_0)}{\sigma(k-1)} = \ft{\sigma(k-1)}{\sigma(k-1)}$ is undefined, so that $\sigma(k) \in \ft{\sigma(a_0)}{\sigma(k-1)}$ is false.
Also, $\dd_k(a_0,k) = \dd(k-1,k) \leq m$ by Proposition \ref{P:initial}(1), so that $\dd_k(a_0,k) = \pi(k-1,k)$.

For $i > 0$, we can determine if $\sigma(k) \in \ft{\sigma(a_i)}{\sigma(k-1)}$ as follows:
\begin{description}
\item[Case 1. $\sigma(k) \in \ft{\sigma(a_{i-1})}{\sigma(k-1)}$]
Note that $i \geq 2$ in this case. 
Since
$|\ft{\sigma(a_{i-1})}{\sigma(k-1)}| < |\ft{\sigma(a_i)}{\sigma(k-1)}|$, so that
$$\ft{\sigma(a_{i-1})}{\sigma(k-1)} \subsetneq \ft{\sigma(a_i)}{\sigma(k-1)}$$
by Lemma \ref{L:ee}(5,6), we have
$\sigma(k) \in \ft{\sigma(a_i)}{\sigma(k-1)}$.

\item[Case 2. $i = 1$ or $\sigma(k) \notin \ft{\sigma(a_{i-1})}{\sigma(k-1)}$]
Let $M:= \max \{ r \in [0,\,i-1] \mid a_{r} > a_i \}$; this is well-defined since $a_0 = k-1 > k-2 \geq a_i$ so that $0 \in \{ r \in [0,\,i-1] \mid a_{r} > a_i \}$.
Let $j \in [a_i+1,\, k-1]$.
We claim that $\sigma(j) \notin \ft{\sigma(a_i)}{\sigma(a_M)}$.
To prove this, we have $j = a_r$ for some $r \in [0,k-2] \setminus \{i\}$.
If $r = M$, then clearly $\sigma(j) = \sigma(a_M) \notin \ft{\sigma(a_i)}{\sigma(a_M)}$, so we need only consider the following five remaining cases.
\begin{description}
\item[Case 2a. $M < r < i$]  We have $j = a_r < a_i$ by definition of $M$, contradicting $j \in [a_i+1,\, k-1]$.

\item[Case 2b. $0 = M < i < r$] We have
\begin{align*}
|\ft{\sigma(a_i)}{\sigma(a_M)}| = |\ft{\sigma(a_i)}{\sigma(k-1)}| < |\ft{\sigma(a_r)}{\sigma(k-1)}| = |\ft{\sigma(a_r)}{\sigma(a_M)}|
\end{align*}
so that $\sigma(a_i) \in \ft{\sigma(a_r)}{\sigma(a_M)}$ and hence $\sigma(j) = \sigma(a_r) \notin \ft{\sigma(a_i)}{\sigma(a_M)}$ by \eqref{E:ee}.

\item[Case 2c. $0 < M < i < r$] We have $\sigma(a_i) \in \ft{\sigma(a_r)}{\sigma(a_M)}$ by Corollary \ref{C:impt}, and so $\sigma(j) = \sigma(a_r) \notin \ft{\sigma(a_i)}{\sigma(a_M)}$ by \eqref{E:ee}.

\item[Case 2d. $0 < r < M < i$] We have $\sigma(a_M) \in \ft{\sigma(a_i)}{\sigma(a_r)}$ by Corollary \ref{C:impt}, and so $\sigma(j) = \sigma(a_r) \notin \ft{\sigma(a_i)}{\sigma(a_M)}$ by \eqref{E:ee}.

\item[Case 2e. $0 = r < M < i$] We have
\begin{align*}
|\ft{\sigma(a_M)}{\sigma(a_r)}| = |\ft{\sigma(a_M)}{\sigma(k-1)}| < |\ft{\sigma(a_i)}{\sigma(k-1)}| = |\ft{\sigma(a_i)}{\sigma(a_r)}|
\end{align*}
so that $\sigma(a_M) \in \ft{\sigma(a_i)}{\sigma(a_r)}$ and hence $\sigma(j) = \sigma(a_r) \notin \ft{\sigma(a_i)}{\sigma(a_M)}$ by \eqref{E:ee}.
\end{description}
This completes the proof of the claim, and so
\begin{align}
\sigma(k) \in \ft{\sigma(a_{i})}{\sigma(a_{M})}
&\Leftrightarrow \pi(a_i,k) - \dd_k(a_M,k) - \dd_{k-1}(a_i,a_M) + \ff^{\bz,I}_{a_i,a_M,k} > 0 \label{E:cond}
\end{align}
by Proposition \ref{P:initial}(2), where the latter can be determined from the knowledge of $\pi$, $\bz$ and induction hypothesis.

It remains to show that
$$
\sigma(k) \in \ft{\sigma(a_{i})}{\sigma(k-1)} \Leftrightarrow \sigma(k) \in \ft{\sigma(a_{i})}{\sigma(a_{M})},
$$
which is clear if $M= 0$ since then $a_M = a_0 = k-1$.
If $M >0$, then $i > i-1 \geq M > 0$, and so
$$
|\ft{\sigma(a_M)}{\sigma(k-1)}| \leq |\ft{\sigma(a_{i-1})}{\sigma(k-1)}| < |\ft{\sigma(a_i)}{\sigma(k-1)}|.
$$
Consequently,
\begin{align}
\ft{\sigma(a_M)}{\sigma(k-1)} &\subseteq \ft{\sigma(a_{i-1})}{\sigma(k-1)}; \label{E:1} \\
\ft{\sigma(a_i)}{\sigma(k-1)} &= \ft{\sigma(a_i)}{\sigma(a_{M})} \cup \{ a_{M} \} \cup \ft{\sigma(a_{M})}{\sigma(k-1)}. \label{E:2}
\end{align}
Since $\sigma(k) \notin \ft{\sigma(a_{i-1})}{\sigma(k-1)}$ by assumption, we have $\sigma(k) \notin \ft{\sigma(a_{M})}{\sigma(k-1)}$ by \eqref{E:1}.
Thus,
\begin{align*}
&\sigma(k) \in \ft{\sigma(a_i)}{\sigma(k-1)}
\Leftrightarrow \\
&\qquad \sigma(k) \in \ft{\sigma(a_i)}{\sigma(k-1)} \setminus (\ft{\sigma(a_{M})}{\sigma(k-1)}
\cup \{\sigma(a_M)\} ) = \ft{\sigma(a_i)}{\sigma(a_M)}
\end{align*}
by \eqref{E:2},
as desired.
The right-hand side of \eqref{E:cond} is therefore a necessary and sufficient condition to check for $\sigma(k) \in \ft{\sigma(a_{i})}{\sigma(k-1)}$.
\end{description}
With the knowledge of whether $\sigma(k) \in \ft{\sigma(a_i)}{\sigma(k-1)}$,
\begin{align*}
\dd_k(a_i,k) = \dd(a_i,k)
&= \dd(a_i,k-1) + \dd(k-1,k) + \bbone_{\sigma(k) \in \ft{\sigma(a_i)}{\sigma(k-1)}} - \ff^{\bz,I}_{a_i,k-1,k} \\
&= \dd_{k-1}(a_i,k-1) + \pi(k-1,k) + \bbone_{\sigma(k) \in \ft{\sigma(a_i)}{\sigma(k-1)}} - \ff^{\bz,I}_{a_i,k-1,k}
\end{align*}
can then be determined from the knowledge of $\pi$, $\bz$ and induction hypothesis.

Thus $\dd_k(a_i,k)$ can be determined from $\pi$ and $\bz$ for all $i \in [0,k-2]$.
Since $a_0,a_1,\dotsc, a_{k-2}$ is a permutation of $1,2,\dotsc, k-1$, this completes the proof that $\dd_k$ can be determined completely by $\pi$ and $\bz$.

Now,
to obtain $A_k$, 
let 
$$
N = 
\begin{cases}
\max\{ i \in [1,\dotsc k-2] \mid \sigma(k) \notin \ft{\sigma(a_i)}{\sigma(k-1)} \}, 
&\text{if } \sigma(k) \notin \ft{\sigma(a_1)}{\sigma(k-1)}; \\
0, &\text{otherwise.} 
\end{cases}
$$
If $N < k-2$, then $\sigma(k) \in \ft{\sigma(a_{N+1})}{\sigma(k-1)}$, or equivalently $|\ft{\sigma(k)}{\sigma(k-1)}| < |\ft{\sigma(a_{N+1})}{\sigma(k-1)}|$.  On the other hand, if $N >0$, then  $\sigma(k) \notin \ft{\sigma(a_N)}{\sigma(k-1)}$, or equivalently $|\ft{\sigma(k)}{\sigma(k-1)}| > |\ft{\sigma(a_{N})}{\sigma(k-1)}|$.
Thus 
\begin{align*}
|\ft{\sigma(a_{k-2})}{\sigma(k-1)}| > \dotsb &> |\ft{\sigma(a_{N+1})}{\sigma(k-1)}| \\
&>  |\ft{\sigma(k)}{\sigma(k-1)}| > |\ft{\sigma(a_{N})}{\sigma(k-1)}| > \dotsb > |\ft{\sigma(a_1)}{\sigma(k-1)}|,
\end{align*}
so that for $i, j \in [1,\, k-1]$, we have
$\sigma(k) \in \ft{\sigma(a_i)}{\sigma(a_j)}$ if and only if $i > N \geq j$ or $N \geq j> i$ or $j > i > N$ by Corollary \ref{C:impt}.
Since $a_1,\dotsc, a_{k-1}$ is a permutation of $1,2,\dotsc, k-1$, we see that $A_k$ is completely determined.

Consequently, by induction, $\dd_k$ and $A_k$ are completely determined by $\pi$ and $\bz$ for all $k \in [2,\,e]$; in particular $\dd = \dd_e$ is completely determined by $\pi$ and $\bz$.
We now claim that
$$T_{\sigma} := \{ (a,b,c) \in [1,\,e]^3 \mid \sigma(c) \in \ft{\sigma(a)}{\sigma(b)} \}$$ can be determined by $\{ A_k \mid k \in [2,\,e]\}$.
Clearly $\sigma(c) \in \ft{\sigma(a)}{\sigma(b)}$ only if $a,b,c$ are distinct,
and when $a,b,c$ are distinct, whether $\sigma(c) \in \ft{\sigma(a)}{\sigma(b)}$ can be determined by whether $\sigma(c') \in \ft{\sigma(a')}{\sigma(b')}$, where $c' =\max\{a,b,c\}$ and $\{a', b', c'\} = \{a,b,c\}$ by \eqref{E:ee}.
In other words, $\sigma(c) \in \ft{\sigma(a)}{\sigma(b)}$ can be determined by whether $(a',b') \in A_{c'}$.
Thus $\pi$ and $\bz$ determine completely $\dd$ and $T_{\sigma}$, so that together with $|\bx|$, they completely determine $\bx$ by Proposition \ref{P:unique-x}.
\end{proof}

We are now ready to prove Proposition \ref{P:unique}.

\begin{proof}[Proof of Proposition \ref{P:unique}]
Let $B'$ be a Scopes-initial block in $\AW_e \DDot{} B$ such that
$(\bz^{B'})^{\sigma_{B'}} = (\bz^B)^{\sigma_B}$ and $\pi_{B'} = \pi_B$.
Note that $\by^{B'} = \by^{\br^*_{B'}}$, $\bz^{B'} = \bz^{\br^*_{B'}}$, $\sigma_{B'} = \sigma_{\by^{B\smash{'}}} = \sigma_{\smash[t]{\by^{\br^*_{B\smash{'}}}}}$, and that $B'$ is Scopes-initial if and only if $\br^*_{B'}$ is $(\II_{B'}, m_{B'})$-initial  (see Remark \ref{R:initial}).
Since $\mv(B') = \mv(B)$ by Proposition \ref{P:moving-vector-Weyl-orbit}, we have $\II_{B'} = \II_B$ and $m_{B'} = m_B$.
Thus $B'$ is Scopes-initial if and only if $\br^*_{B'}$ is $(\II_B,m_B)$-initial.
By Theorem \ref{T:unique-initial}, $\br^*_{B'}$ is completely determined by $|\br^*_{B'}| = |\br^{w_0}|$, $\pi^{\II_{B}, m_B}_{\br^*_{B'}} = \pi^{\II_{B'}, m_{B'}}_{\br^*_{B'}} = \pi_{B'} = \pi_B$ and $(\bz^{\br^*_{B'}})^{\sigma_{\by^{(\smash{\br^*_{B'}})}}} = (\bz^{B'})^{\sigma_{B'}} = (\bz^B)^{\sigma_{B}}$.
Since $\br^*_{B'}$ completely determines $\core(B')$ by \eqref{E:bu} and $\wt(B') = |\mv(B')| = |\mv(B)| = \wt(B)$ by \eqref{E:wt-mv}
, there is at most one such block $B'$ by Theorem \ref{T:Naka}.  Hence, there is at most one Scopes-initial block in $\AW_e \DDot{} B$ with the same Scopes vector and the same pyramid numbers as $B$.
\end{proof}

\subsection{RoCK blocks} \label{S:RoCK}

In this subsection we relate the classification of Scopes equivalence classes obtained in Corollary \ref{C: } to Webster's notion of RoCK blocks for Ariki-Koike algebras~\cite{Webster}, and give a simple description of RoCK blocks in terms of their pyramid numbers.
We also provide a closed formula for the number of Scopes inequivalent RoCK blocks in a $\AW_e$-orbit.

Webster's RoCK blocks generalise the Rouquier blocks first defined for symmetric groups and Schur algebras in \cite{CT}.  
These blocks have since been developed for other algebras such as the group algebras of finite general linear groups, Iwahori-Hecke algebras of type $A$ and $q$-Schur algebras.  

In \cite{Lyle-RoCK}, Lyle defines a block $B$ of $\HH_n$ to be Rouquier if and only if there exists $s \in \ZZ/e\ZZ$ such that for all $\bl = (\lambda^{(1)},\dotsc, \lambda^{(\ell)}) \in \PP^{\ell}$ lying in $B$, we have $\fs(\C^{(i)}_{j+1}) \geq \fs(\C^{(i)}_j) + \wt_e(\lambda^{(i)}) -1$ for all $i \in [1,\, \ell]$ and $j \in [1,\, e-1]$, where $\quot_e(\beta_{r_i+s}(\lambda^{(i)})) = (\C^{(i)}_1,\dotsc, \C^{(i)}_e)$. 

By \cite[Proposition~4.6]{Webster}, a block $B$ of $\HH_n$ is a RoCK block if and only if it is Scopes equivalent to a Rouquier block. 
Webster also provided a description of RoCK blocks of Ariki-Koike algebras in \cite[Subsection 4.2]{Webster}.
As a consequence of our classification of Scopes equivalence classes (Corollary \ref{C: }), we obtain a particularly concise characterisation for the RoCK property in terms of pyramid numbers, as follows.

\begin{thm} \label{T:RoCK}
Let $B$ be a block of $\HH_n$. The following statements are equivalent:
\begin{enumerate}
\item 
$B$ is a RoCK block;
\item $\pi_B(a,b) = m_B$ for all $a,b \in [1,\,e]$ with $a< b$.
\item $\pi_B(j,j+1) = m_B$ for all $j \in [1,\,e-1]$.
\end{enumerate}
\end{thm}

\begin{proof} 
Suppose first that $m_B = 0$.
By Lemma \ref{L:delta-basic}(1), we have $\dd_B(a,b) \geq 0$, so that $\pi_B(a,b) = \min(\dd_B(a,b), m_B) = \min(\dd_B(a,b), 0) = 0 = m_B$ for all $a,b \in [1,\,e]$ with $a< b$.  
On the other hand, $B$ is always a RoCK block by \cite[Theorem 3.17 and Corollary 4.27]{LQT}.
Thus, the theorem holds for $m_B =0$.

Since (2) implies (3) is trivial, it remains to show that $(1) \Rightarrow (2)$ and $(3) \Rightarrow (1)$ when $m_B > 0$.

\begin{description}
\item[$(1) \Rightarrow (2)$]

Suppose that $B$ is a RoCK block with $m_B > 0$.
Then $B$ is Scopes equivalent to a Rouquier block $C$ by \cite[Proposition~4.6]{Webster}, and we may further assume 
by \cite[Proposition 3.18]{Lyle-RoCK} that $C$ has the property that for all $\bl = (\lambda^{(1)}, \dotsc, \lambda^{(\ell)}) \in \PP^{\ell}$ lying in $C$, we have $$\fs(\B^{(i)}_{j+1}) \geq \fs(\B^{(i)}_j) +N$$ for all $i \in [1,\, \ell]$ and $j \in [1,\, e-1]$, where $\quot_e(\beta_{r_i+s}(\lambda^{(i)})) = (\B^{(i)}_1,\dotsc,\B^{(i)}_e)$, $s \in \ZZ/e\ZZ$ and $N \gg 0$ .
Let
\begin{align*}
\beta_{\br^{w_0}}(\bl^{w_0})  &= (\B^{(1)},\dotsc, \B^{(\ell)}), \\
\quot_e(\UU(\beta_{\br^{w_0}}(\bl^{w_0})) &= (\C_1,\dotsc, \C_e), \\
\quot_e(\B^{(a)}) &= (\C^{(a)}_1, \dotsc, \C^{(a)}_e)
\end{align*}
for all $a \in [1,\,\ell]$.
Then $\fs(\C_j) = \sum_{a=1}^{\ell} \fs(\C^{(a)}_j)$ by \cite[Lemma 2.6(2)]{LT}.

Note that if $w_0 = \tau\bt $ where $\bt = (t_1,\dotsc, t_{\ell}) \in \ZZ^{\ell}$ and $\tau \in \sym{\ell}$, then
$$\B^{(a)} = \beta_{r_{\tau(a)} + et_{a}}(\lambda^{(\tau(a))}) = (\beta_{r_{\tau(a)}+s}(\lambda^{(\tau(a)})))^{+et_a-s},$$
so that
\begin{align*}
(\B^{(\tau(a))}_1, \dotsc, \B^{(\tau(a))}_e) &= \quot_e(\beta_{r_{\tau(a)+s}}(\lambda^{(\tau(a))}))
= \quot_e((\B^{(a)})^{+(-et_a+s)}) \\
&= ((\C^{(a)}_{e-s+1})^{+(-t_a+1)}, \dotsc, (\C^{(a)}_{e})^{+(-t_a+1)},
(\C^{(a)}_{1})^{+(-t_a)}, (\C^{(a)}_{e-s})^{+(-t_a)})
\end{align*}
by \eqref{E:quot}.
Let $j \in [1,\, e-1]$.  Then 
\begin{align*}
N \leq \fs(\B^{(\tau(a))}_{j+1}) - \fs(\B^{(\tau(a))}_{j})
&= \fs((\C^{(a)}_{j+1+e-s})^{+(-t_a + \bbone_{j+1 \leq s})}) -  \fs((\C^{(a)}_{j+e-s})^{+(-t_a + \bbone_{j \leq s})}) \\
&= \fs(\C^{(a)}_{j+1+e-s}) - \fs(\C^{(a)}_{j+e-s}) - \delta_{js}
\end{align*}
where, here and hereafter, the subscripts $j+1+e-s$ and $j+e-s$ are to be read modulo $e$.  
Thus
\begin{align*}
x^C_{j+1+e-s} - x^C_{j+e-s} &= \fs(\C_{j+1+e-s}) - \fs(\C_{j+e-s}) \\
&= \sum_{a=1}^{\ell} \fs(\C^{(a)}_{j+1+e-s}) - \sum_{a=1}^{\ell} \fs(\C^{(a)}_{j+e-s}) \\
&\geq (N+\delta_{js})\ell,
\end{align*}
so that $y^C_{j+1+e-s} - y^C_{j+e-s} \geq N+ \delta_{js}$.
Since this holds for all $j \in [1,\, e-1]$, we thus have 
$$
y^C_{e-s+1} < \dotsb < y^C_e < y^C_1 < \dotsb < y^C_{e-s},
$$
so that 
$\sigma_C = \rr_e^{e-s}$, and
$$
\dd_C(j,j+1) = y^C_{j+1+e-s} - y^C_{j+e-s} - \delta_{js} + \bbone_{z^C_{j+e-s} \leq \Ht_{\II_C}(z^C_{j+1+e-s})} \geq N.
$$
Since $N \gg 0$, we may assume that $N \geq m_C$.
Consequently, $\pi_C(a,b) = m_C$ for all $a,b \in [1,\,e]$ with $a< b$ by Lemma \ref{L:delta-basic}(2) (and induction on $b-a$).
Since $B$ and $C$ are Scopes equivalent, we thus have $\pi_B(a,b) = \pi_C(a,b) = m_C = m_B$ for all $a,b \in [1,\,e]$ with $a< b$ by Corollary \ref{C: }.


\item[$(3) \Rightarrow (1)$]
Suppose that $\pi_B(j,j+1) = m_B>0$ for all $j \in [1,\,e-1]$.
We shall construct a Rouquier block $R$ which is Scopes equivalent to $B$; this will show that $B$ is a RoCK block by \cite[Proposition 4.6]{Webster}.

Let $\bt = (t_1,\dotsc, t_e) \in \ZZ^e$ such that $|\bt| = 0$ and $t_{j+1} > t_j$ for all $j \in [1,\, e-1]$.
Then $\bt \in \AW_e$ by \eqref{E:affineWeylgroup}.
Let $R := \bt \sigma_B^{-1} \DDot{} B \in \AW_e \DDot{} B$.
Then $\mv(R) = \mv(B)$ 
by Proposition \ref{P:moving-vector-Weyl-orbit}, so that $m_R = m_B$ and $\II_{R} = \II_B$.
Furthermore, by Lemma \ref{L:cs_j}(1), we have
\begin{align*}
\by^{R} &= \bt + (\by^B)^{\sigma_B}; \\
\bz^{R} &= (\bz^B)^{\sigma_B}.
\end{align*}
Thus, for $j \in [1,\,e-1]$, we have
$y^{R}_{j+1} - y^{R}_j = (y^B_{\sigma_B(j+1)} - y^B_{\sigma_B(j)}) + (t_{j+1} - t_j) > 0$
since $y^B_{\sigma_B(j+1)} \geq y^B_{\sigma_B(j)}$ by definition of $\sigma_B = \sigma_{\by^B}$.
In particular, 
$y^{R}_1 < y^{R}_2 < \dotsb < y^{R}_e$, so that $\sigma_{R} = 1_{\sym{e}}$.
Hence, $(\bz^{R})^{\sigma_{R}} = \bz^{R} = (\bz^B)^{\sigma_B}$, and for $j \in [1,\,e-1]$,
\begin{align*}
\dd_{R}(j,j+1) &= y^{R}_{j+1} - y^{R}_j + \bbone_{z^{R}_j \leq \Ht_{\II_R} (z^{R}_{j+1})} \\
&= (y^B_{\sigma_B(j+1)} + t_{j+1}) - (y^B_{\sigma_B(j)}  + t_j) + \bbone_{z^B_{\sigma_B(j)} \leq \Ht_{\II_B} (z^B_{\sigma_B(j+1)})} \\
&= \dd_B(j,j+1) + (t_{j+1} -t_j) + \bbone_{\sigma_B(j) > \sigma_B(j+1)} \\
&> \dd_B(j,j+1) \geq \pi_B(j,j+1) = m_B.
\end{align*}
%
Consequently, by Theorem \ref{T:equiv1}, 
any $\bl = (\lambda^{(1)},\dotsc, \lambda^{(\ell)}) \in \PP^{\ell}$ lying in $R$ has no addable node with $(e,\br)$-residue $j$, or equivalently, every $\lambda^{(i)}$ has no addable node with $(e,r_i)$-residue $j$.
Clearly, for any $\mu \in \PP$ lying in the same block $B_i$ of $\HH_{\FF, q, r_i}(|\lambda^{(i)}|)$ as $\lambda^{(i)}$, the $\ell$-partition $(\lambda^{(1)}, \dotsc, \lambda^{(i-1)}, \mu , \lambda^{(i+1)}, \dotsc, \lambda^{(\ell)})$ also lies in $R$, and so has no addable node with $(e,\br)$-residue $j$; consequently, $\mu$ also has no addable node with $(e,r_i)$-residue $j$.
We can therefore apply Theorem \ref{T:equiv1} to $B_i$ to get 
$$\fs(\C^{(i)}_{j+1}) - \fs(\C^{(i)}_j) \geq \wt(B_i) = \wt_e(\lambda^{(i)}),$$ where $\quot_e(\beta_{r_i}(\lambda^{(i)})) = (\C^{(i)}_1,\dotsc, \C^{(i)}_e)$.
Since this is true for all $i \in [1,\,\ell]$ and $j \in [1,\,e-1]$, we see that $R$ is a Rouquier block.  

To show that $B$ is Scopes equivalent to $R$, we use the characterisation of Scopes equivalence classes in  Theorem \ref{T:Scopes}.  Since $R$ and $B$ lie in the same $\AW_e$-orbit and $(\bz^R)^{\sigma_R} = (\bz^B)^{\sigma_B}$, it remains to show that $\pi_R = \pi_B$.
But since $m_B \geq 1$, we see from Lemma \ref{L:delta-basic}(2) (and by induction) that
\begin{align*}
m_B \leq \dd_B(a,a+1) \leq \dd_B(a,a+2) \leq \dotsb \leq \dd_B(a,b)
\end{align*}
for all $a, b \in [1,\,e]$ with $a < b$, and a similar statement also holds for $R$.
Thus, $$\pi_R(a,b) = m_R = m_B = \pi_B(a,b)$$ for all $a, b \in [1,\,e]$ with $a < b$, and the proof is complete.
\end{description}
\end{proof}

We conclude this subsection with the number of Scopes inequivalent RoCK blocks in a $\AW_e$-orbit.  
Since this has been addressed in \cite[Proposition 4.25]{LQT} for core blocks, we shall only do so for the non-core ones, i.e.\ those with $m_B > 0$.

We begin with demonstrating the existence of a RoCK block in any $\AW_e$-orbit of a non-core block.  The following result actually proves much more.

\begin{prop} \label{P:RoCK-exist}
Let $B$ be a block of $\HH_n$ with $m_B > 0$.
For any $\tau \in \sym{e}$, there exists a RoCK block $R \in \AW_e \DDot{} B$ with $(\bz^R)^{\sigma_R} = (\bz^B)^{\tau}$.
\end{prop}

\begin{proof}
Let $\by = (y_1,\dotsc, y_{e}) \in \ZZ^e$ be such that $|\by| = |\by^B|$ and for each $j \in [1,\, e-1]$, $y_{j+1} - y_j \geq m_B+1$.
Then $|\by - (\by^B)^{\tau}| = |\by| - |(\by^B)^{\tau}| = |\by| - |\by^B| = 0$, so that $\by - (\by^B)^{\tau} \in \AW_e$ by \eqref{E:affineWeylgroup}.
Let $R = ((\by - (\by^B)^{\tau}) \tau^{-1}) \DDot{} B \in \AW_e \DDot{} B$.  Then $m_R = m_B$ by Proposition \ref{P:moving-vector-Weyl-orbit}, and  
\begin{align*}
\by^R &= ((\by - (\by^B)^{\tau}) (\tau)^{-1}) \CDot{1} \by^B = 
(\by - (\by^B)^{\tau}) + (\by^B)^{\tau} = \by; \\
\bz^R &= ((\by - (\by^B)^{\tau}) (\tau)^{-1}) \CDot{0} \bz^B = (\bz^B)^{\tau}
\end{align*}
by Lemma \ref{L:cs_j}(1).
Thus, $y^R_{j+1} - y^R_j = y_{j+1} - y_j \geq m_B+1$ for all $j \in [1,\,e-1]$, giving both $\sigma_R = 1_{\sym{e}}$ and for all $j \in [1,\,e-1]$, $\dd_R(j,j+1) \geq m_B$, and hence $$\pi_R(j,j+1) = \min(\dd_R(j,j+1), m_R) = \min(\dd_R(j,j+1), m_B) = m_B = m_R.$$  
Thus $R$ is a RoCK block by Theorem \ref{T:RoCK}.
\end{proof}

\begin{cor}[cf. {\cite[Theorem A]{Webster}}]
Let $B$ be a block of $\HH_n$ with $m_B > 0$.  For each $i \in \ZZ/\ell\ZZ$, let $$k_i = |\{ j \in [1,\,e] \mid z^B_j = i \}|.$$  Then the number of Scopes inequivalent RoCK blocks in $\AW_e \DDot{} B$ is exactly
$$
\frac{e!}{k_0!\, k_1!\, \dotsm \, k_{\ell-1}!}.
$$
\end{cor}

In particular, the number of Scopes inequivalent RoCK blocks in $\AW_e \DDot{} B$ depends only on the multiplicities of the $\ell$-residues occurring in $\bz^B$.

\begin{proof}
If $R$ is a RoCK block in $\AW_e \DDot{} B$, then its Scopes vector is a rearrangement of $\bz^B$ by Corollary \ref{C:w-z}.
On the other hand, each rearrangement of $\bz^B$ is the Scopes vector of a RoCK block in $\AW_e \DDot{} B$ by Proposition \ref{P:RoCK-exist}.  

Since Scopes equivalent blocks in $\AW_e \DDot{} B$ are characterised by their Scopes vectors and pyramid numbers by Theorem \ref{T:Scopes}, and that all RoCK blocks in $\AW_e \DDot{} B$ have the same pyramid numbers by Theorem \ref{T:RoCK}, we see that the number of Scopes inequivalent RoCK blocks in $\AW_e \DDot{} B$ is equal to the number of distinct Scopes vectors of RoCK blocks in $\AW_e \DDot{} B$, which in turn is the number of distinct rearrangements of $\bz^B$.  The corollary now follows.
\end{proof}


\end{document}